\documentclass[11pt]{amsart}
\usepackage[T1]{fontenc}

\usepackage{amssymb, amsthm}
\usepackage[all]{xy}
\usepackage{pb-diagram, pb-xy}

\usepackage{hyperref}
\hypersetup{colorlinks=true} 

\newtheorem{thm}{Theorem}[section]
\newtheorem*{mainthm}{Main Theorem}

\newtheorem{cor}[thm]{Corollary}

\newtheorem{lem}[thm]{Lemma}
\newtheorem{prop}[thm]{Proposition}

\theoremstyle{definition}
\newtheorem{defn}[thm]{Definition}

\newtheorem{ex}[thm]{Example}
\newtheorem{para}[thm]{--}

\numberwithin{equation}{thm}

\newcommand{\IA}{\mathbb A}
\newcommand{\IC}{\mathbb C}
\newcommand{\IG}{\mathbb G}
\newcommand{\IP}{\mathbb P}
\newcommand{\IQ}{\mathbb Q}
\newcommand{\IR}{\mathbb R}
\newcommand{\IZ}{\mathbb Z}

\newcommand{\bA}{\mathbf A}
\newcommand{\bM}{\mathbf M}

\newcommand{\bAff}{\mathbf{Aff}}
\newcommand{\bHo}{\mathbf{Ho}}
\newcommand{\bSet}{\mathbf{Set}}
\newcommand{\bVec}{\mathbf{Vec}}

\DeclareMathOperator{\End}{End}
\DeclareMathOperator{\Hom}{Hom}
\DeclareMathOperator{\spec}{Spec}

\newcommand{\bad}{\mathrm{b}}
\newcommand{\eff}{\mathrm{eff}}
\newcommand{\gud}{\mathrm{g}}
\newcommand{\id}{\mathrm{id}}
\newcommand{\Q}{\mathrm{Q}}
\newcommand{\Qaff}{\mathrm{Q}_{\mathrm{aff}}}
\newcommand{\Qeff}{\mathrm{Q}_{\mathrm{eff}}}
\newcommand{\Qh}{\mathrm{Q_h}}
\renewcommand{\top}{\mathrm{top}}
\newcommand{\eps}{\varepsilon}

\renewcommand{\leq}{\leqslant}
\renewcommand{\geq}{\geqslant}

\title{Algebraic cogroups and Nori motives}
\author{Javier Fres\'an} 
\address{Javier Fres\'an, CMLS, \'Ecole Polytechnique, F-91128 Palaiseau, France}
\email{javier.fresan@polytechnique.edu}

\author{Peter Jossen}%
\address{Peter Jossen, ETH Z\"urich, D-MATH, R\"amistrasse 101, CH-8092 Z\"urich, Switzerland}
\email{peter.jossen@math.ethz.ch}

\begin{document}

\begin{abstract}
We introduce the notion of algebraic cogroup over a subfield $k$ of the complex numbers and use it to prove that every Nori motive over $k$ is isomorphic to a quotient of a motive of the form $H^n(X, Y)(i)$.
\end{abstract}

\maketitle

\tableofcontents

\section{Introduction and overview}

\begin{par}
Let us fix a subfield $k$ of the complex numbers and denote by $\bM(k)$ the $\IQ$-linear neutral tannakian category of mixed motives over $k$ introduced by Nori. We call objects of $\bM(k)$ simply  \emph{motives}. A rather immediate consequence of the construction of Nori's category is that every motive is isomorphic to a subquotient of a motive 
\begin{equation}\label{Eqn:SumOfElementaryMotivesIntro}
\bigoplus_\alpha H^{n_\alpha}(X_\alpha,Y_\alpha)(i_\alpha)
\end{equation}
associated with a finite collection $(X_\alpha)_\alpha$ of varieties over $k$, closed subvarieties $Y_\alpha\subseteq X_\alpha$, and integers $n_\alpha$ and $i_\alpha$. The question at the outset of this paper is whether every motive may be written as a \emph{submotive} or as a \emph{quotient} of a motive of the shape \eqref{Eqn:SumOfElementaryMotivesIntro}. We can now give a positive answer to it, see Theorems \ref{Thm:SubquotientProblemMainQuotients} and \ref{Thm:SubquotientProblemMainSubs}.
\end{par}

\begin{par}
\begin{mainthm} Let $M$ be a motive over $k$. There exists a pair of varieties $(X_0,Y_0)$ defined over $k$, integers $n_0$ and $i_0$, and an injective morphism of motives
$$M \xrightarrow{\quad} H^{n_0}(X_0,Y_0)(i_0).$$
Dually, there exists a pair of varieties $(X_1, Y_1)$ over $k$, integers $n_1$ and $i_1$, and a surjective morphism of motives
$$H^{n_1}(X_1,Y_1)(i_1) \xrightarrow{\quad} M.$$
\end{mainthm}
\end{par}

\begin{par} As we will explain in \ref{Par:Sum101IsElementary}, after introducing motives, a sum of the form \eqref{Eqn:SumOfElementaryMotivesIntro} is actually isomorphic to a single motive $H^{n_0}(X_0,Y_0)(i_0)$, so there is no difference between the question at the outset of the paper and the seemingly stronger answer of the main theorem.
\end{par}

\vspace{4mm}
\begin{para}\label{Par:ApplicationAbstractSetting}
A reason why the reader may be interested in a statement like this is the following: Nori's formalism is very efficient in producing and comparing exact functors on $\bM(k)$. For example, suppose we are given two exact functors
$$F\colon \bM(k) \to \bA \qquad\qquad  G\colon \bM(k) \to \bA$$
with values in a $\IQ$-linear abelian category $\bA$ and a natural injection $\varepsilon_M \colon F(M)\to G(M)$ for each motive $M$. If $\varepsilon_M$ is an isomorphism for all objects of the form $H^n(X,Y)(i)$, then $\varepsilon$ is an isomorphism of functors. Indeed, it suffices to present an arbitrary motive as a subquotient of \eqref{Eqn:SumOfElementaryMotivesIntro}, then write down the obvious short exact sequences. If however $F$ and $G$ are only left exact functors, this argument breaks down, and there seems to be no formal remedy. To the rescue comes our main theorem, according to which we can inject an arbitrary motive $M$ into a motive $M_0$ of the form $H^n(X,Y)(i)$. The obvious short exact sequence to write down is
$$\begin{diagram}
\node{0}\arrow{e}\node{F(M)}\arrow{s,l,J}{\varepsilon_M}\arrow{e}\node{F(M_0)}\arrow{s,l,J,A}{\varepsilon_{M_0}}\arrow{e}\node{F(M_0/M)}\arrow{s,l,J}{\varepsilon_{M_0/M}}\\
\node{0}\arrow{e}\node{G(M)}\arrow{e}\node{G(M_0)}\arrow{e}\node{G(M_0/M)}
\end{diagram}$$
where the vertical map $\varepsilon_{M_0}$ is an isomorphism by assumption. A straightforward diagram chase reveals that $\varepsilon_M$ is an isomorphism as well, and since $M$ was arbitrary, $\varepsilon$ is an isomorphism of functors.
\end{para}

\vspace{4mm}
\begin{para}
\begin{par}
Here are concrete examples of the situation described in \ref{Par:ApplicationAbstractSetting}. First and foremost, the functor associating with a motive $M$ the $\IQ$-vector space $$F(M) = \Hom_{\bM(k)}(\IQ(0),M)$$ is left exact, and there is a natural injection of $F(M)$ into 
$$
G(M) = \Hom_{\mathbf{MHS}}(\IQ(0), \mathrm{R_{Hdg}}(M)),
$$ where $\mathbf{MHS}$ stands for the category of mixed Hodge structures and $\mathrm{R_{Hdg}} \colon \bM(k) \to \mathbf{MHS}$ for the Hodge realisation functor. To prove that $F \to G$ is an isomorphism in this case would be a kingly achievement. Unfortunately, the only thing we can say in this direction is that it would ``suffice'' to prove it for motives of the form $H^n(X, Y)(i)$. 
\end{par}

\begin{par}
Things are brighter in relative situations. Let us consider Nori motives over the function field $\IC(t)$. To define the category, it is better not to embed this field into $\IC$, though we could. Given a pair of varieties $(X,Y)$ over $\IC(t)$, we can choose a model $(\mathcal X, \mathcal Y)$ of it over the affine line~$\IA^1$. The relative cohomology $\mathcal H^n((\mathcal X, \mathcal Y)/\IA^1, \IQ)(i)$ is a constructible sheaf on~$\IA^1$, whose nearby fibre at $0\in \IA^1$ is independent of the choice of the model. We may use
$$[X,Y,n, i] \mapsto \mathcal H^n((\mathcal X, \mathcal Y)/\IA^1, \IQ)(i)_{\vec 0}$$
as our standard cohomology theory to construct motives over $\IC(t)$. The resulting category $\bM(\IC(t))$ comes then with a functor associating with a motive $M$ a local system $\mathcal L_M$ on some non-empty Zariski open subset $U_M$ of $\IA^1(\IC)$. The functor 
$$G\colon  M \longmapsto \Gamma(U_M, \mathcal L_M)$$
from the category of motives over $\IC(t)$ to the category of finite-dimensional rational vector spaces is left exact. Whenever a motive $M_0$ over $\IC(t)$ comes via base change from a motive over $\IC$, its associated local system is constant, given by the Betti realisation $\mathrm{R_B}(M_0).$ Thus, if $c(M) \subseteq M$ denotes the largest submotive of $M$ which comes via base change from a motive over $\IC$, then 
$$F\colon M \longmapsto \mathrm{R_B}(c(M))$$
is a left exact functor, and there is a natural injection $\varepsilon_M \colon F(M) \to G(M)$. A \emph{theorem of the fixed part} for motives over $\IC(t)$ would state that this injection is an isomorphism for all motives. Thanks to our main theorem, we only need to check that $\varepsilon_M$ is an isomorphism when $M$ is of the form $H^n(X,Y)(i)$, and in this situation a geometrical construction saves the day: we can again choose a model $(\mathcal X, \mathcal Y)$ of $(X,Y)$, regard this model as a pair of varieties over $\IC$, and consider the motive $M_0=H^n(\mathcal X, \mathcal Y)(i)$ and the canonical morphism $M_0 \to M$ induced by $X \to \mathcal X \times_\IC \spec(\IC(t))$. Now one can check that the image of $M_0 \to M$ is $c(M)$ and that the image of $\mathcal L_{M_0} \to \mathcal L_M$ are the invariants of $\mathcal L_M$.  
\end{par}
\begin{par}
Another, admittedly quite similar example is the case where we work with \emph{exponential motives}, as introduced in \cite{FJ}, instead of ordinary motives. Exponential motives come equipped with a perverse realisation, and there is a largest constant sheaf inside the perverse realisation of an exponential motive. On the other hand, an exponential motive contains a largest ordinary submotive, whose perverse realisation is constant. Our main theorem holds verbatim for exponential motives, and we can apply it in the same way. We postpone these applications to a future paper.
\end{par}
\end{para}

\vspace{4mm}
\begin{para}
Let us now explain what algebraic cogroups have to do with the main theorem. We focus here, and also later in the paper, on the problem of writing motives as quotients of sums of the form \eqref{Eqn:SumOfElementaryMotivesIntro}. We start trying to solve some particular cases. First, suppose a morphism of pairs of varieties $f \colon (X,Y) \to (X_1,Y_1)$ over $k$ is given, and consider the induced morphism of motives
$$
f^\ast\colon H^n(X_1,Y_1) \xrightarrow{\quad} H^n(X,Y).
$$
A not very difficult geometric construction (see Proposition \ref{Pro:KernelIsElementaryQuotient}) produces a pair $(X_2,Y_2)$ and an exact sequence of motives 
$$H^n(X_2,Y_2) \xrightarrow{\quad} H^n(X_1,Y_1) \xrightarrow{\:\:f^\ast\:\:} H^n(X,Y)$$
which exhibits the kernel of $f^\ast$ as a quotient of the motive $H^n(X_2,Y_2)$. Consider then the next best scenario, where a second morphism $g\colon (X,Y) \to (X_1,Y_1)$ is given, and we want to show that the kernel of a $\IZ$-linear combination such as
$$af^\ast + bg^\ast\colon H^n(X_1,Y_1) \xrightarrow{\quad} H^n(X,Y)$$
is a quotient of some motive $H^n(X_2,Y_2)$. Here we are at a loss at first sight, since there is no apparent way of recognising $af^\ast + bg^\ast$ as the morphism of motives induced by a morphism of varieties $af+bg$. Unless, of course, the pair $(X,Y)$ has in some way the structure of a cogroup. An arbitrary pair $(X,Y)$ will not admit such a structure. However, in algebraic topology we learn that the suspension of any topological space admits a canonical cogroup structure, so we should try to define a suspension $\Sigma(X,Y)$ together with a cogroup structure and a natural isomorphism of motives
$$
H^{n+1}(\Sigma(X,Y)) \cong H^n(X,Y).
$$
This is indeed possible, allowing us to show that kernels of linear combinations are quotients of motives of the form \eqref{Eqn:SumOfElementaryMotivesIntro}. This does not yet prove the main theorem for arbitrary motives, but solves the essential case.
\end{para}

\vspace{14mm}
\section{Cogroups in algebraic topology}

\begin{par}
This section serves as a reminder about cogroups in the classical setting of algebraic topology, if not for the convenience of the reader, at least as a rehearsal for the authors. What we call \emph{cogroup} here is also refered to as co-H-group,  H-cogroup, or homotopy cogroup in the literature. A detailed treatment of cogroups can be found \textit{e.g.} in Chapter 2 of~\cite{Ark11}. Throughout, we denote by $\bHo$ the category whose objects are pointed CW-complexes and whose morphisms are homotopy classes of base-point preserving continuous maps (the homotopies are also required to respect base points). All cohomology groups are understood to be with rational coefficients.  
\end{par}

\vspace{4mm}
\begin{para}
A \textit{cogroup} is a group object in the opposite category $\bHo^{\mathrm{op}}$. For this to make sense, it is of course necessary that the category $\bHo^{\mathrm{op}}$ has a product, which it indeed has. It is given by the coproduct in $\bHo$, which is the wedge product of pointed spaces
$$X \vee Y = (X\sqcup Y)/x_0\sim y_0$$
where $x_0$ and $y_0$ are the base points of $X$ and $Y$ respectively. In more detail, a cogroup is a pointed CW-complex $X$, together with maps
$$c\colon X \to X\vee X \qquad  i\colon X\to X$$
called \textit{comultiplication} and \textit{inversion}. As a counit serves the map to a point $e \colon X\to \{x_0\}$. We denote by $0$ the constant map $X\to X$ sending every element of $X$ to the base point. The following diagrams of continuous maps, expressing coassociativity and  the role of the inversion and counit, are required to commute up to homotopy: 
\begin{equation}\label{Eqn:CogroupAxiomDiagrams}
\begin{array}{cc}
\begin{diagram}
\node{X}\arrow{s,l}{c}\arrow{e,t}{c}\node{X\vee X}\arrow{s,r}{\id\vee c}\node{X}\arrow{s,l}{c}\arrow{se,=}\arrow{e,t}{c}\node{X\vee X}\arrow{s,r}{\id\vee e}\\
\node{X\vee X}\arrow{e,t}{c\vee\id}\node{X\vee X\vee X}\node{X\vee X}\arrow{e,t}{e\vee \id}\node{X}
\end{diagram}\\[15mm]
\begin{diagram}
\node[2]{X\vee X}\arrow{e,t}{\id\vee i}\node{X\vee X}\arrow{se,t}{\mathrm{fold}}\\
\node{X}\arrow[3]{e,t}{0}\arrow{ne,t}{c}\arrow{se,b}{c}\node[3]{X.}\\
\node[2]{X\vee X}\arrow{e,t}{i\vee \id}\node{X\vee X}\arrow{ne,b}{\mathrm{fold}}
\end{diagram}
\end{array}
\end{equation}

In the last diagram, we have also used the fold-map $X\vee X \to X$, which is the map given by the identity on both copies of $X$ in $X\vee X$. Its role is dual to the diagonal map of a space into its self-product. A cogroup is said to be \textit{cocommutative} if, moreover, the following diagram commutes up to homotopy: 
$$\begin{diagram}
\node[2]{X}\arrow{sw,t}{c}\arrow{se,t}{c}\\
\node{X\vee X}\arrow[2]{e,t}{\mathrm{flip}}\node[2]{X\vee X.}
\end{diagram}$$
\end{para}

\vspace{4mm}
\begin{para}\label{Par:TopologicalCogroupAndYoneda}
Let $X = (X,c,i)$ be a cogroup. For every pointed espace $Y$, the set $\Hom_{\bHo}(X,Y)$ of homotopy classes of maps from $X$ to $Y$ carries a group structure. The neutral element is the zero map, and the composition law is defined by declaring that the sum of two morphisms $f\colon X\to Y$ and $g\colon X \to Y$ is the morphism
\begin{equation}\label{Eqn:DefOfSumOfMorphisms}
X \xrightarrow{\:\:c\:\:}X\vee X\xrightarrow{\:\:f\vee g\:\:}Y\vee Y\xrightarrow{\:\: \mathrm{fold}\:\:}Y
\end{equation}
which is customarily written additively as $f+g$, despite the fact that the group operation need not be commutative. It is a standard consequence of Yoneda's lemma that to give a cogroup structure on $X$ amounts to giving a factorisation of the functor
$$\Hom_\bHo(X,-) \colon \bHo \xrightarrow{\quad} \bSet$$
through the category of groups. A justification for the arguably strange convention to write the group law on $\Hom_\bHo(X,Y)$ additively might be that this group structure is compatible with cohomology, in the sense that the maps between relative cohomology groups
\begin{equation}\label{Eqn:CogroupAndCohomologyOperation}
f^\ast + g^\ast \colon H^n(Y,y_0) \to H^n(X,x_0) \qquad (f + g)^\ast\colon H^n(Y,y_0) \to H^n(X,y_0) 
\end{equation}
are equal. This is shown \textit{e.g.} in Proposition 2.2.9 (5) of \cite{Ark11}. An easy consequence of this fact is that cogroups are connected. Indeed, taking for $f$ and $g$ the identity on $X$, we see that the multiplication-by-2 map on $H^0(X,x_0)$ is induced by a map $X \to X$ on the one hand, and on the other hand any automorphism of $H^0(X,x_0)$ induced by a map $X \to X$ is in an appropriate basis given by a permutation of basis vectors, hence has finite order. It follows that $H^0(X,x_0)$ is trivial, hence $X$ is connected.
\end{para}

\vspace{4mm}
\begin{para}\label{Par:S1TopologicalCogroup}
A non-trivial example of a cogroup is the circle $S^1$. The comultiplication is the pinch map $S^1 \to S^1\vee S^1$ and the inversion is the map $S^1\to S^1$ of degree $-1$, reversing the orientation of the circle. The factorisation of the functor
$$\Hom_\bHo(S^1,-) \colon \bHo \xrightarrow{\quad}\bSet$$
through the category of groups is the functor associating with a pointed CW-complex its fundamental group. In particular, $S^1$ is \emph{not} cocommutative. The equality of maps \eqref{Eqn:CogroupAndCohomologyOperation} amounts to the fact that the Hurewicz map
$$\pi_1(Y,y_0) \to H^1(Y, y_0)$$
is a group homomorphism. More generally, the suspension $\Sigma X = S^1\wedge X$ of any CW-complex $X$ admits a cogroup structure. The comultiplication on $\Sigma X$ is the map
$$\Sigma X = S^1\wedge X \xrightarrow{\:\mathrm{pinch}\wedge\id\:} (S^1\vee S^1)\wedge X = (S^1\wedge X) \vee (S^1\wedge X) = \Sigma X \vee \Sigma X$$
and the inversion is similarly induced from the inversion of the circle. A particular example of this construction is the cogroup structure on the higher-dimensional spheres \hbox{$S^{n+1} = \Sigma^nS^1$.} These are cocommutative for $n \geq 1$ and, more generally, for any space $X$ the double suspension $\Sigma^2 X$ is a cocommutative cogroup. All cogroups we will encounter later are suspensions of some space. It is surprisingly difficult to give examples of cogroups which are not suspensions, but they do exist, see \cite{BH86}. 
\end{para}

\vspace{4mm}
\begin{para}
We can transport these ideas to the category of pointed algebraic varieties in a straightforward fashion. It is not too clear what it means for two morphisms between two algebraic varieties to be homotopic, so let us for the time being just work with complex algebraic varieties, and say that morphisms between varieties are homotopic if the corresponding continuous maps between topological spaces of complex points are so. We may now call, naively, a pointed complex algebraic variety, together with algebraic maps
$$c\colon X \to X\vee X \qquad  i\colon X\to X$$
an algebraic cogroup if the topological space $X(\IC)$ together with the comultiplication $c$ and inversion $i$ is a cogroup.
\end{para}

\vspace{4mm}
\begin{prop}
There exist no non-trivial naive algebraic cogroups.
\end{prop}

\begin{proof}
Let $X$ be a complex algebraic variety with a fixed base point, write $e \colon X\to \ast$ for the structural morphism, and suppose that there is a morphism of algebraic varieties $c\colon X \to X\vee X$ such that the composite maps
$$f_1\colon X \xrightarrow{\:\:c\:\:} X\vee X \xrightarrow{\:\:\id\vee e\:\:} X\qquad f_2\colon X \xrightarrow{\:\:c\:\:} X\vee X \xrightarrow{\:\:e\vee\id\:\:} X$$
are homotopic to the identity. We will show that $X$ is contractible. In fact, a composite of the homotopy equivalences $f_{\alpha_n}\cdots f_{\alpha_2}f_{\alpha_1}$ is the trivial map $X \to \ast \to X$. Let $X_1,X_2,\ldots, X_n$ be the irreducible components of $X$. The morphism of algebraic varieties $c$ sends each irreducible component of $X$ into one of the irreducible components of $X\vee X$, which are of the form $X_i\vee \ast$ or $\ast \vee X_i$. Therefore, each irreducible component of $X$ is sent to the base point by $f_1$ or by $f_2$. Let's say $f_{\alpha_1}$ sends $X_1$ to the base point. The component $X_2$ is then sent by $f_{\alpha_1}$ into one of the irreducible components of $X$, hence we can find $f_{\alpha_2}$ such that $f_{\alpha_2}f_{\alpha_1}$ sends $X_1\cup X_2$ to the base point. Proceeding this way, we find that indeed a composition $f_{\alpha_n}\cdots f_{\alpha_2}f_{\alpha_1}$ is the trivial map.
\end{proof}

\vspace{14mm}%%%%%%%%%%%%%%%%%%%%%%%%%%%%%%%%%%%%%%%%%%%%%%%%%%%%
\section{Elementary homotopy theory with pairs of varieties}%%%%%
%%%%%%%%%%%%%%%%%%%%%%%%%%%%%%%%%%%%%%%%%%%%%%%%%%%%%%%%%%%%%%%%%

\begin{par}
Let us fix a subfield $k$ of $\IC$, and convene that all varieties and morphisms of varieties are defined over $k$. We say that two morphisms of algebraic varieties $f$ and $g$ from $X$ to $Y$ are \textit{homotopic} if the induced continuous maps from $X(\IC)$ to $Y(\IC)$ are homotopic. We do not require that the homotopy between these continuous maps is in any way induced by a morphism of algebraic varieties. Similarly, we say that a morphism of algebraic varieties $f\colon X\to Y$ is a \textit{homotopy equivalence} if the induced continuous map $f\colon X(\IC) \to Y(\IC)$ is a homotopy equivalence. Again, we do not ask for a homotopy inverse which is in any way induced by a morphism of varieties from $Y$ to $X$. There are many examples of homotopy equivalences between algebraic varieties which do not have a homotopy inverse which is given by an algebraic map. For instance, the morphism of algebraic varieties
\begin{equation}\label{Eqn:JouanolouExample}
\{\mbox{idempotent $2\times 2$-matrices of rank 1}\}\to \IP^1    
\end{equation}
sending $A$ to the line generated by the columns of $A$ is a homotopy equivalence, since its fibres are torsors under affine spaces. This map cannot have an algebraic homotopy inverse since the space of idempotent matrices of rank $1$ is affine, morphisms from connected projective varieties to affine varieties are constant, and $\IP^1$ is not contractible. Similarly, the morphism of affine varieties 
$$(\IA^1 \sqcup \IA^1)/_{0 \sim 0, 1\sim 1} \xrightarrow{\quad} \IA^1/_{0\sim 1}$$
contracting one of the lines to a point is a homotopy equivalence, and yet has no homotopy inverse which is given by a morphism of algebraic varieties. 
\end{par}

\vspace{4mm}
\begin{para}
We call \emph{pair of affine algebraic varieties} any pair $(X,Y)$ consisting of an affine variety $X$ over $k$ and a closed, non-empty subvariety $Y$ of $X$. With the obvious notion of morphisms and compositions, these pairs of algebraic varieties form a category $\bAff_2(k)$. Since $k$ is embedded in the complex numbers, there is a well-defined functor
\begin{equation}\label{Eqn:AnalytificationRaw}
\begin{array}{rcl}
(-)^\top\colon \bAff_2(k) & \to & \bHo\\  
(X,Y) & \mapsto & (X(\IC)/Y(\IC), Y(\IC)/Y(\IC))
\end{array}
\end{equation}
where $\bHo$ stands as in the previous section for the category of pointed CW-complexes and continuous maps up to homotopy. This functor is compatible with cohomology, in the sense that there is a canonical and natural isomorphism of vector spaces
\begin{equation}\label{Eqn:AnalytificationAndCohomology}
H^n(X(\IC), Y(\IC); \IQ) \cong H^n(X(\IC)/Y(\IC), Y(\IC)/Y(\IC); \IQ)
\end{equation}
for each pair $(X,Y)$. Indeed, the cohomology of the pair $(X(\IC),Y(\IC))$ is by construction the cohomology of the mapping cone of the inclusion $i\colon Y(\IC) \to X(\IC)$. Since $i$ is the inclusion of a closed algebraic subvariety, it is with respect to an appropriate CW-structure the inclusion of a finite, closed CW-complex into a finite CW-complex, hence is in particular a cofibration, and we know that the mapping cone of a cofibration is homotopic to the corresponding quotient space.
\end{para}

\vspace{4mm}
\begin{para}
We say that a morphism $f$ in $\bAff_2(k)$ is a \emph{homotopy equivalence} if the induced continuous map $f^\top$ between pointed topological spaces is, and we say that two morphisms $f$ and $g$ with same source and target are \emph{homotopic to each other} if the continuous maps $f^\top$ and $g^\top$ are. Being homotopic is an equivalence relation on morphisms in $\bAff_2(k)$ which is compatible with compositions, so there is a well-defined category $\bAff_2(k)_\simeq$ whose objects are pairs of affine algebraic varieties over $k$, and whose morphisms are homotopy classes of morphisms of pairs of algebraic varieties. By definition, the functor \eqref{Eqn:AnalytificationRaw} induces a functor
\begin{equation}\label{Eqn:AnalytificationUpToHomotopy}
(-)^\top\colon \bAff_2(k)_\simeq \to \bHo
\end{equation}
which sends homotopy equivalences to isomorphisms. As we have pointed out at the beginning of this section, the functor \eqref{Eqn:AnalytificationUpToHomotopy} is not conservative: it sends morphisms which are not isomorphisms to isomorphisms. Forcing \eqref{Eqn:AnalytificationUpToHomotopy} to be conservative leads to the following definition: 
\end{para}

\vspace{4mm}
\begin{defn}
We call \emph{category of algebraic homotopy types over $k$} and denote by $\bHo(k)$ the localisation in the class of homotopy equivalences of the category $\bAff_2(k)_\simeq$.
\end{defn}

\vspace{4mm}
\begin{para}
By construction, the functor $(-)^\top: \bAff_2(k) \to \bHo$ induces a functor
$$(-)^\top: \bHo(k) \to \bHo$$
which is faithful and conservative. Morphism sets in $\bHo(k)$ between any two objects contain a distinguished element, the zero morphism, and are in particular non-empty. Denoting by $\ast = \spec(k)$ the variety with one point, the zero morphism $(X,Y) \to (X',Y')$ is the composite morphism (the variety $Y'$ may have no rational points)
$$(X,Y) \xrightarrow{\:\:\:\: } (\ast, \ast) \xrightarrow{\:\:s^{-1}\:\:} (Y',Y') \xrightarrow{\:\:\subseteq\:\:} (X', Y')$$
where $s$ is the homotopy equivalence $(Y', Y') \to (\ast, \ast)$. As we shall explain, the class of homotopy equivalences in $\bAff_2(k)_\simeq$ admits a calculus of left fractions. Consequently every morphism $(X,Y)\to (X_1,Y_1)$ in $\bHo(k)$ can be represented by a roof $s^{-1}f$
$$\begin{diagram}
\node[3]{(\widetilde X,\widetilde Y)}\\
\node{(X,Y)}\arrow{nee,t}{f}\node[3]{(X_1,Y_1)}\arrow{nw,t}{s}
\end{diagram}$$
where $s$ is a homotopy equivalence and $f$ an arbitrary morphism of pairs. It is obvious that the class of homotopy equivalences in $\bAff_2(k)_\simeq$ contains all isomorphisms, is stable under composition, and admits left cancellation: given morphisms of pairs
\begin{displaymath}
\xymatrix{(\widetilde X,\widetilde Y) \ar[rr]^{s} & & (X,Y) \ar@/_.5pc/[rr]_{f} \ar@/^.5pc/[rr]^{g} & &(X_1,Y_1)}
\end{displaymath}
where $s$ is a homotopy equivalence, then $f$ is homotopic to $g$ if $fs$ is homotopic to $gs$. What remains to check in order to show that homotopy equivalences admit a left calculus of fractions is the left Ore condition, which is the statement of Proposition \ref{Pro:LeftOreCondition} below. We refer the reader to \cite[III, \S2]{GM96} for details about calculus of fractions in a category. Our choice to consider only pairs of affine varieties is merely for convenience, so we can embedd every variety $X$ into a contractible variety $\IA^N$, and has otherwise no big consequences. Indeed, Jouanolou's trick provides for every quasiprojective variety $X$ a homotopy equivalence $\widetilde X \to X$ where $\widetilde X$ is affine, see \cite[Lemme 1.5]{Jou73}. An elementary example of this trick is given by \eqref{Eqn:JouanolouExample}.
\end{para}

\begin{par}
In algebraic topology, various instances of spheres are used, each having its distinct advantages. For example, the boundary of the standard $(n+1)$-simplex $\partial \Delta^{n+1}$ as well as the pair of spaces $(\Delta^n, \partial \Delta^n)$ are models for the $n$-sphere. These models are very useful for comparing homology and homotopy groups, notably in the proof of Hurewicz's Theorem. Other models of the $n$-sphere are the $n$-fold suspension $\Sigma^nS^0$ of the $0$-sphere $S^0$ which is useful when dealing with long exact homotopy sequences, and the set of vectors of unit length in $\IR^{n+1}$, useful when dealing with differential structures. If we are only interested in spaces up to homotopy, the set of nonzero vectors in $\IR^{n+1}$ can also be used as a sphere. For the circle, we have on top of all these possibilities also the models $\IR/\IZ$ and $\IC^\times/\IR_{>0}$, which are particularly suitable when dealing with the group structure on the circle. Similarly, there are many choices for a model of the circle in $\bHo(k)$.
\end{par}

\begin{defn}
We call \emph{standard algebraic circle} the object $S^1(k) = (\IA^1,\{0,1\})$ of $\bHo(k)$.
\end{defn}

\begin{par}
There is a canonical isomorphism $S^1(k)^\top \cong [0,1]_{0 \sim 1} = S^1$ in $\bHo$. Marking any other two distinct rational points on $\IA^1$ results in an isomorphic object, but the isomorphism becomes canonical only once an ordering of the two marked points is chosen. Yet another model of the circle is the algebraic $n$-gon, which is the affine variety obtained by gluing $n$ copies of $\IA^1$ in the obvious way, marked at any point. A particular case is the $1$-gon
$$\spec(\{f\in k[t]\:|\:\: f(0)=f(1)\})$$
marked at the singular point. If the base field $k$ admits a quadratic extension, we could also mark $\IA^1$ in the zeroes of a nonsplit quadratic polynomial. The resulting object in $\bHo(k)$ is not isomorphic to the standard circle, but is obviously a \emph{form} of it. Also, we could pick any two quadratic polynomials $f$ and $g$, and consider the square $f(x)g(y)=0$ embedded in the plane $\IA^2$. This square becomes isomorphic to the standard circle over a biquadratic extension of $k$. We don't know how to classify all Galois forms of the standard circle in $\bHo(k)$. Lastly, we could also think of $(\IG_m, 1)$ as a circle. It is not isomorphic to the standard circle over any base field, and is for our purposes the wrong choice.
\end{par}

\vspace{4mm}
\begin{defn}
Let $(X,Y)$ and $(X',Y')$ be pairs of algebraic varieties. We call 
$$(X,Y) \vee (X',Y') = (X \sqcup X', Y\sqcup Y')$$
the \emph{wedge}, and 
$$(X,Y) \wedge (X',Y') = (X \times X', (X\times  Y')\cup (Y\times X'))$$
the \emph{smash} product of $(X,Y)$ and $(X',Y')$. We call the pairs
\begin{eqnarray*}
 C(X,Y) & =  & (X,Y)\wedge (\IA^1,\{0\}) = (X\times \IA^1, (X\times \{0\})\cup (Y\times \IA^1)) \\
 \Sigma(X,Y) & = & (X, Y) \wedge (\IA^1 , \{0,1\}) = (X\times \IA^1, (X\times \{0,1\})\cup (Y\times \IA^1))  
\end{eqnarray*}
the \emph{cone} and the \emph{suspension} of $(X,Y)$.
\end{defn}

\vspace{4mm}
\begin{para}\label{Par:Puppe}
Cones and suspensions do what they ought to do. They are compatible with the functor $(-)^\top$ in the obvious way. There is a natural morphism $(X,Y) \to C(X,Y)$, namely the one sending $x$ to $(x,1)$. The cone $C(X,Y)$ is contractible, in the sense that the unique morphism $C(X,Y) \to (\ast,\ast)$ is a homotopy equivalence. The triple of pairs
$$(X,Y) \to C(X,Y) \to \Sigma(X,Y)$$
induces a long exact sequence of cohomology groups, which, since $C(X,Y)$ is contractible, degenerates to isomorphisms $H^n(\Sigma(X,Y)) \cong H^{n+1}(X,Y)$. Indeed, the morphism of pairs
\begin{equation}\label{Eqn:HReplacement}
(X,Y) \to ((X \times\{0,1\}) \cup (Y \times\IA^1), (X\times\{0\}) \cup (Y\times \IA^1))    
\end{equation}
sending $x$ to $(x,1)$ is a homotopy equivalence, hence an isomorphism in $\bHo(k)$. The long exact sequence of cone and suspension can be identified with the the long exact sequence of the following triple of spaces.
\begin{equation}\label{Eqn:ConeSuspensionTriple}
(X\times\{0\}) \cup (Y\times \IA^1) \:\: \subseteq\:\: (X \times\{0,1\}) \cup (Y \times\IA^1) \:\:\subseteq\:\: X \times \IA^1
\end{equation}
Given a triple of nonempty varieties $Z\subseteq Y \subseteq X$, we should look at the morphisms $(Y,Z)\to (X,Z) \to (X,Y)$ as a cofibre sequence. Associated with it comes the long Puppe-sequence
\begin{equation}\label{Eqn:PuppeSequence}
(Y,Z)\to (X,Z) \to (X,Y) \to \Sigma(Y,Z) \to \Sigma(X,Z) \to \cdots 
\end{equation}
where the connecting morphism $(X,Y) \to \Sigma(Y,Z)$ needs some explanation. It is a morphism in the category $\bHo(k)$, given by the roof
\begin{equation}\label{Eqn:PuppeConnectorRoof}
\begin{diagram}
\node[3]{(X \times \IA^1, X\cup_ZY)}\\
\node{(X,Y)}\arrow{nee,t}{f}\node[3]{\Sigma(Y,Z)}\arrow{nw,t}{s}
\end{diagram}
\end{equation}
where $X\cup_ZY$ stands for the union $(X\times\{1\}) \cup(Z\times\IA^1) \cup(Y\times\{0\})$, in which the topologist may recognise a homotopy pushout, as we explain in the proof of Proposition \ref{Pro:LeftOreCondition}. The map $f$ in \eqref{Eqn:PuppeConnectorRoof} sends $x$ to $(x,0)$, and the homotopy equivalence $s$ is defined by the inclusion $Y \times \IA^1 \subseteq X \times \IA^1$. For later reference, we observe that the diagram of vector spaces
\begin{equation}\label{Eqn:PuppeConnectorInCohomology}
\begin{diagram}
\node{H^n(Y,Z)}\arrow{s,l}{\cong}\arrow{e,t}{\partial}\node{H^{n+1}(X,Y)}\arrow{s,r}{s^\ast}\\
\node{H^{n+1}(\Sigma(Y,Z))}\arrow{e,t}{f^\ast}\node{H^{n+1}(X\times\IA^1,X\cup_ZY)}
\end{diagram}
\end{equation}
commutes, so the connecting morphism $\partial$ in the long exact sequence is the composite of the suspension isomorphism and the morphism induced in cohomology by $s^{-1}f$.
\end{para}

\vspace{4mm}
\begin{lem}\label{Lem:ImmersionHomotopyFactorisation}
Any morphism of pairs of affine algebraic varieties $f\colon (X,Y) \to (X_1,Y_1)$ can be factorised as 
\begin{displaymath}
\xymatrix{
(X, Y) \ar[rr]^f \ar[rd]_{g} & & (X_1, Y_1) \\
& (\widetilde X_1,\widetilde Y_1) \ar[ur]_s &}
\end{displaymath}
%$$(X,Y) \xrightarrow{\:\:g\:\:} (\widetilde X_1,\widetilde Y_1) \xrightarrow{\:\:s\:\:} (X_1,Y_1)$$
where $g$ is a closed immersion, and $s$ is a homotopy equivalence which admits an algebraic homotopy inverse. 
\end{lem}

\begin{proof}
Since $X$ is affine, there exists a closed immersion $e \colon X\to \IA^N$ for some sufficiently large $N$. Consider the pair $(\widetilde X_1, \widetilde Y_1) = (\IA^N\times X_1, \IA^N\times Y_1)$, and define the morphisms $g$ and~$s$ by $g(x) = (e(x),f(x))$ and $s(t,x_1)=x_1$. The morphism $g$ is a closed immersion because $e$ is, the morphism $s$ is a homotopy equivalence because $\IC^N$ is contractible, and $f = sg$ holds by construction. An algebraic homotopy inverse to $s$ is the morphism $(X_1,Y_1) \to (\widetilde X_1,\widetilde Y_1)$ sending $x_1$ to $(0,x_1)$.
\end{proof}

\vspace{4mm}
\begin{lem}\label{Lem:PushoutOfClosedImmersions}
Let $f_1\colon (X,Y) \to (X_1,Y_1)$ and $f_2\colon (X,Y) \to (X_2,Y_2)$ be morphisms of pairs of algebraic varieties, both of them given by closed immersions. The pushout
\begin{equation}\label{Eqn:PushoutOfClosedImmersions}
\begin{diagram}
\node{(X,Y)}\arrow{s,l}{f_1}\arrow{e,t}{f_2}\node{(X_2, Y_2)}\arrow{s,l}{g_2}\\
\node{(X_1,Y_1)}\arrow{e,t}{g_1} \node{(X', Y')}
\end{diagram}
\end{equation}
exists in the category of pairs of varieties. If $X_1$ and $X_2$ are affine, then so is $X'$. 
\end{lem}

\begin{proof}
We define $X'$ to be the variety obtained by glueing $X_1$ and $X_2$ along $X$. There is an obvious way of gluing ringed spaces, and one can check that gluing varieties along a common closed subvariety results is a variety, as is done in \cite[Theorem 3.3 and Corollary 3.7]{Sch05}. If $X = \spec(A)$ and $X_i=\spec(A_i)$ are all affine, then $X'$ is affine, and indeed the spectrum of the ring
$$A' = \{(a_1,a_2) \in A_1\times A_2 \:|\:\: \varphi_1(a_1)=\varphi_2(a_2)\}$$
where $\varphi_i\colon A_i\to A$ are the surjective ring morphisms corresponding to the inclusions \hbox{$X\to X_i$.} There are canonical closed immersions $g_i\colon X_i\to X'$, and with these, the diagram
$$\begin{diagram}
\node{X}\arrow{s,l}{f_1}\arrow{e,t}{f_2}\node{X_2}\arrow{s,r}{g_2}\\
\node{X_1}\arrow{e,t}{g_1} \node{X' = X_1\cup_XX_2\hspace{-22mm}}
\end{diagram}$$
commutes and is a push-out diagram in the category of varieties. As a closed subvariety $Y' \subseteq X'$ we choose the union $g_1(Y_1) \cup g_2(Y_2)$. With this choice for $Y'$ we obtain a diagram of the shape \eqref{Eqn:PushoutOfClosedImmersions}, and the universal property of the pair $(X',Y')$ in the category of pairs of varieties is a direct consequence of the universal property of $X'$ in the category of varieties. 
\end{proof}

\vspace{4mm}
\begin{prop}\label{Pro:LeftOreCondition}
Let $s: (X,Y) \to (X_1, Y_1)$ be a homotopy equivalence between affine pairs and let $f\colon (X,Y) \to (X_2, Y_2)$ be an arbitrary morphism of affine pairs of varieties. There exists a diagram of affine pairs of varieties
\begin{equation}\label{Eqn:HomotopyPushAlgebraic}
\begin{diagram}
\node{(X,Y)}\arrow{s,l}{f}\arrow{e,t}{s}\node{(\widetilde X, \widetilde Y)}\arrow{s,r}{\widetilde f}\\
\node{(X_1,Y_1)}\arrow{e,t}{\widetilde s}\node{(\widetilde X_1, \widetilde Y_1)}
\end{diagram}
\end{equation}
which commutes up to homotopy, and where $\widetilde s$ is a homotopy equivalence.
\end{prop}

\begin{proof}
\begin{par}
We will construct the homotopy pushout of $f$ and $s$. Let us first recall how this is done classically in algebraic topology, see \cite[\S4.2]{Die08}: Given a diagram of pointed CW-complexes and continuous maps
\begin{equation}\label{Eqn:HomotopyPushTopological}
\begin{diagram}
\node{X}\arrow{s,l}{f}\arrow{e,t}{s}\node{\widetilde X}\\
\node{X_1}
\end{diagram}
\end{equation}
the \emph{homotopy pushout} is the space $\widetilde X_1$ obtained by quotienting $X_1 \sqcup (X \wedge [0,1]) \sqcup \widetilde X$ by the relations $(x,0)\sim f(x)$ and $(x,1)\sim s(x)$ for $x\in X$. If $f$ and $s$ are both cofibrations the homotopy pushout is an ordinary pushout, but not in general: the homotopy pushout remembers the space of relations $X$. For example, if $X_1$ and $X_2$ are reduced to a point, then the homotopy pushout is the suspension of $X$. If the morphism $s:X\to \widetilde X$ in \eqref{Eqn:HomotopyPushTopological} is a homotopy equivalence, then so is the canonical morphism $\widetilde s: X_1 \to \widetilde X_1$.
\end{par}
\begin{par}
We now model the homotopy pushout with pairs of algebraic varieties. By Lemma \ref{Lem:ImmersionHomotopyFactorisation} we can and will suppose without loss of generality that the morphisms $s:X\to \widetilde X$ and $f\colon X\to X_1$ in \eqref{Eqn:HomotopyPushAlgebraic} are closed immersions. Define the pair $(\widetilde X_1, \widetilde Y_1)$ to be the ordinary pushout of the diagram
\begin{equation}\label{Eqn:HomotopyPushAlgebraic2}
\begin{diagram}
\node{(X,Y) \vee (X,Y)}\arrow{s,l}{f\vee s}\arrow{e,t}{(i_0, i_1)}\node{(X \times \IA^1, Y\times \IA^1)}\arrow{s}\\
\node{(X_1,Y_1)\vee(\widetilde X, \widetilde Y) }\arrow{e,t}{u}\node{(\widetilde X_1, \widetilde Y_1)}
\end{diagram}
\end{equation}
where $(i_0,i_1)$ is the map sending one copy of $X$ to $X\times\{0\}$ and the other to $X\times\{1\}$. Since both, $(i_0,i_1)$ and $f\vee s$ are closed immersions, this pushout exists as we have seen in Lemma \ref{Lem:PushoutOfClosedImmersions}. We define $\widetilde s$ and to be the restriction of the map $u$ to the component $(X_1,Y_1)$, and $\widetilde f$ to be the restriction of $u$ to $(\widetilde X, \widetilde Y)$. In other terms, $(\widetilde X_1, \widetilde Y_1)$ is the quotient space
$$(\widetilde X_1, \widetilde Y_1) = \big((X_1,Y_1) \sqcup (X\times\IA^1, Y\times \IA^1) \sqcup (\widetilde X, \widetilde Y))/_{(x,0)= s(x), (x,1)\sim f(x)}
$$
and the morphisms $\widetilde s$ and $\widetilde f$ are the ones induced by inclusions. The functor $(-)^\top: \bHo(k) \to \bHo$ sends the diagram \eqref{Eqn:HomotopyPushAlgebraic} to a homotopy pushout. In particular, the morphism $\widetilde s$ is a homotopy equivalence since $s$ is so by hypothesis.
\end{par}
\end{proof}

\vspace{4mm}
\begin{cor}
The class of homotopy equivalences in $\bAff_2(k)/_\simeq$ admits a calculus of left fractions. In the terminology of \cite[III, \S2, Def. 6]{GM96}, the class of homotopy equivalences is \emph{localising}.
\end{cor}

\begin{proof}
All conditions of \cite[III, \S2, Def. 6]{GM96} are trivially satisfied by homotopy equivalences in $\bAff_2(k)/_\simeq$, except condition (b) which is the content of Proposition \ref{Pro:LeftOreCondition}.
\end{proof}

%%%%%%%%%%%%%%%%%%%%%%%%%%%%%%%%%%%%%
\vspace{14mm}%%%%%%%%%%%%%%%%%%%%%%%%
\section{Algebraic cogroups}%%%%%%%%%
%%%%%%%%%%%%%%%%%%%%%%%%%%%%%%%%%%%%%

\begin{par}
We now have all ingredients at our disposal in order to define algebraic cogroups, which is what we will do in this section. All varieties and morphisms of varieties are tacitly supposed to be defined over a fixed field $k\subseteq \IC$. We abbreviate $\ast = \spec(k)$.
\end{par}

\vspace{4mm}
\begin{defn}
An \emph{algebraic cogroup} consists of a pair of affine algebraic varieties $(X,Y)$ together with two morphisms in $\bHo(k)$
$$c\colon (X,Y)\to (X,Y)\vee (X,Y) \qquad i\colon (X,Y)\to (X,Y)$$
such that the diagrams analogous to \eqref{Eqn:CogroupAxiomDiagrams} commute in $\bHo(k)$.
\end{defn}

\vspace{4mm}
\begin{para}
Let $(X,Y)$ be an algebraic cogroup. For every pair $(X_1,Y_1)$, the set
$$\Hom_{\bHo(k)}((X,Y), (X_1,Y_1))$$
is equipped with a canonical group structure, which is functorial in $(X_1,Y_1)$ for morphisms of pairs, and functorial in $(X,Y)$ for morphisms of cogroups.  The sum of two morphisms $f$ and $g$ from $(X,Y)$ to $(X_1,Y_1)$ is defined as in \eqref{Eqn:DefOfSumOfMorphisms}, hence is compatible with the functor $(-)^\top: \bHo(k) \to \bHo$. In particular we have 
$$(f+g)^\top = f^\top + g^\top\qquad \mbox{and}\qquad (-f)^\top = -(f^\top)$$
and the relation with cohomology \eqref{Eqn:CogroupAndCohomologyOperation} holds verbatim for algebraic cogroups.
\end{para}

\vspace{4mm}
\begin{para}
The standard circle $S^1(k) = (\IA^1, \{0,1\})$ is a cogroup. The comultiplication or pinch map can be given as
$$c\colon (\IA^1, \{0,1\}) \xrightarrow{\:\:=\:\:} (\IA^1, \{0,\tfrac 12, 1\})\xleftarrow{\:\:\simeq\:\:} (\IA^1, \{0,1\}) \vee (\IA^1, \{0,1\})$$
where the map labelled with $=$ is the morphism of pairs given by the identity of $\IA^1$, and the map $\simeq$ is the homotopy equivalence given by $t \mapsto \tfrac 12t$ on the first component, and by $t \mapsto \tfrac 12(1+t)$ on the second component. The inversion morphism is the morphism of pairs
$$i\colon (\IA^1, \{0,1\}) \to (\IA^1, \{0,1\})$$
sending $t$ to $1-t$. The functor $(-)^\top$ sends $S^1(k)$ to the  circle $S^1$ with its usual cogroup structure. 
\end{para}

\vspace{4mm}
\begin{prop}
Let $(X, Y)$ be a pair of affine algebraic varieties. The suspension $\Sigma(X, Y)$ admits the structure of an algebraic cogroup. 
\end{prop}

\begin{proof}
The suspension of $(X,Y)$ is the wedge $S^1(k)\wedge (X,Y)$, and the smash and wedge product satisfy the distributivity property
$$((S^1(k) \vee S^1(k)) \wedge (X,Y) = (S^1(k) \wedge (X,Y)) \vee (S^1(k) \wedge (X,Y))$$
so that a comultiplication on $\Sigma(X,Y)$ is induced by the comultiplication on $S^1(k)$ just as in topology, as explained in \ref{Par:S1TopologicalCogroup}.
\end{proof}

\vspace{4mm}
\begin{ex}\label{Exa:WedgeOfCirclesCogroup}
Let $\Lambda$ be a finitely generated free $\IZ$-module and let $\alpha\colon \Lambda\to \Lambda$ be a $\IZ$-linear endomorphism. We can construct a pair of varieties $(X,Y)$, an endomorphism $f\colon (X,Y)\to(X,Y)$ in $\bHo(k)$, and an isomorphism $\varphi\colon(\Lambda\otimes\IQ) \to H^1(X,Y)$ such that the diagram
\begin{equation}\label{Eqn:DiagramLinearToCohomological}
\begin{diagram}
\node{\Lambda\otimes\IQ}\arrow{s,l}{\varphi}\arrow{e,t}{\alpha\otimes 1}\node{\Lambda\otimes\IQ}\arrow{s,l}{\varphi}\\
\node{H^1(X,Y)}\arrow{e,t}{f^\ast}\node{H^1(X,Y)}
\end{diagram}
\end{equation}
commutes.  Let us choose a $\IZ$-basis $\lambda_1, \ldots, \lambda_n$ of $\Lambda$, and consider the pair $(X_0, Y_0)$ where $X_0$ is a disjoint union of $n+1$ points $x_0, \ldots, x_n$ and $Y$ consists of the single point $x_0$. For integers $1\leq i,j \leq n$ we can consider the endomorphism $e_{ij}:(X_0,Y_0) \to (X_0,Y_0)$ sending $x_i$ to $x_j$ and sending all other points of $X_0$ to $x_0$. The induced map in cohomology $e_{ij}^\ast:H^0(X_0,Y_0)\to H^0(X_0,Y_0)$ corresponds via the obvious isomorphism $H^0(X_0,Y_0) \cong \Lambda\otimes\IQ$ to the linear map $\eps_{ij}$ sending $\lambda_i$ to $\lambda_j$ and sending all other basis vectors to zero. Taking for $(X,Y)$ the suspension of $(X_0,Y_0)$, we end up with an isomorphism 
$$\varphi\colon (\Lambda\otimes \IQ)\cong H^0(X_0,Y_0) \cong H^1(\Sigma(X_0,Y_0))  = H^1(X,Y)$$
and the maps $\Sigma e_{ij}\colon (X,Y)\to (X,Y)$ induce endomorphisms of $H^1(X,Y)$ which correspond via $\varphi$ to the linear maps $\eps_{ij}$. The given endomorphism $\alpha\colon \Lambda\to \Lambda$ is a $\IZ$-linear combination, say
$$\alpha = \sum_{1\leq i,j \leq n}a_{ij}\eps_{ij}$$
of the endomorphisms $\pi_{ij}$. We use the cogroup structure on  $(X,Y) = \Sigma(X_0,Y_0)$ to produce an endomorphism
$$f = \sum_{1\leq i,j \leq n}a_{ij}.(\Sigma e_{ij})$$
of $(X,Y)$ in the category $\bHo(k)$. The endomorphism $f$ depends on the order of summation, because the cogroup structure on $(X,Y)$ is not commutative. The map $f^\ast \colon H^1(X,Y)\to H^1(X,Y)$ induced by $f$ is independent of this choice, and makes the diagram \eqref{Eqn:DiagramLinearToCohomological} commute.
\end{ex}

\vspace{4mm}
\begin{para}
Nothing stands in the way of defining \emph{algebraic homotopy groups} of a pair $(X,Y)$ as
$$\pi_n(X,Y) = \Hom_{\bHo(k)}(S^n(k),(X,Y))$$
where $S^n(k)$ is the standard sphere $\Sigma^{n-1} S^1(k)$. These are groups, commutative for $n\geq 2$, and there are long exact Puppe sequences for triples. The functor $(-)^\top$ defines an injection
$$\pi_n(X,Y) \to \pi_n((X,Y)^\top)$$
which need not be an isomorphism. Indeed, we have $\pi_1(\IG_{m,k},\{1\})=0$ while $\pi_1(\IC^\times, 1) \cong \IZ$. We did not pursue this any further.
\end{para}

%%%%%%%%%%%%%%%%%%%%%%%%%%%%%%%%%%%%%%%%
\vspace{14mm}%%%%%%%%%%%%%%%%%%%%%%%%%%%
\section{Variants of Nori's quiver}%%%%%
%%%%%%%%%%%%%%%%%%%%%%%%%%%%%%%%%%%%%%%%

\begin{par}
We explain the basic mechanism of Nori's quiver representation, and present some easy variants. For the whole section, we fix a subfield $k$ of $\IC$. All varieties are understood to be varieties over $k$, and all vector spaces are understood to be finite-dimensional rational vector spaces. Given a vector space $V$, we set
$$V(-i) = V\otimes H^1(\IG_m, 1)^{\otimes i}$$
with the usual convention that, for negative integers $i$, the $i$-th tensor power means the $(-i)$-th tensor power of the linear dual.
\end{par}

\vspace{4mm}
\begin{para}
By a quiver, we understand a class of objects and a class of morphisms, together with maps that attach to each morphism its source and its target, and to each object an identity morphism. Nori's quiver is the quiver $\Q(k)$ whose objects are tuples
$$q = [X,Y,n,i]$$
consisting of a quasiprojective variety $X$, a closed subvariety $Y$, and two integers $n$ and $i$ that will be called the \emph{degree} and the \emph{twist}. There are three types of morphisms in $\Q(k)$. All morphisms with target $[X,Y,n,i]$ are:
\begin{enumerate}
    \item[(a)] for every morphism $(X,Y) \to (X',Y'),$ a morphism $[X',Y',n,i] \to [X,Y, n, i]$, 
    
    \smallskip
    
    \item[(b)] for every closed subvariety $Z \subseteq Y,$ a morphism $[Y,Z,n-1,i] \to [X,Y,n,i]$,
    
    \smallskip
    
    \item[(c)] one morphism $[X \times \IG_m, X\times\{1\}\cup Y \times\IG_m, n+1,i+1] \to [X,Y,n,i]$.
\end{enumerate}
We will refer to these morphisms as morphisms of type (a), (b) and (c). With Nori's quiver $\Q(k)$ comes a \emph{representation}
$$\rho: \Q(k) \to \bVec_\IQ$$
associating with an object $q = [X,Y,n,i]$ the finite-dimensional rational vector space
$$\rho(q) = \rho([X,Y,n,i]) = H^n(X,Y)(i)$$
and associating with morphisms of type (a) the usual induced morphism in cohomology, with morphisms of type (b) the connecting morphism in the long exact sequence of the triple $Z \subseteq Y \subseteq X$, and with morphisms of type (c) the K\"unneth isomorphism.  
\end{para}

\vspace{4mm}
\begin{para}
\begin{par}
The ring $\End(\rho)$ has the structure of a $\IQ$-proalgebra: It is the limit of finite-dimensional $\IQ$-algebras
$$\End(\rho) = \lim_{Q \subseteq \Q(k)}\End(\rho|_Q)$$
where the limit ranges over the finite subquivers $Q$ of $\Q(k)$, with restrictions as transition maps. Nori's category of mixed motives $\bM(k)$ is the category of finite-dimensional, continuous $\End(\rho)$-modules. It is abelian, $\IQ$-linear, with a forgetful functor to $\bVec_\IQ$, called the \emph{Betti realisation}. The algebra $E$ acts contiuously on $\rho(q) = H^n(X,Y)(i)$ for every object $q = [X,Y,n,i]$ of $\Q(k)$, so when we speak about the \emph{motive} $H^n(X,Y)(i)$ we mean the vector space $H^n(X,Y)(i)$ together with its canonical $\End(\rho)$-module structure. 
\end{par}
\begin{par}
Nori's category features two interesting properties which are not immediate from the definition: First, it constitutes a universal cohomology theory for all cohomology theories that can be compared with singular cohomology, and second, there is a canonical tensor structure on the category $\bM(k)$ making it a neutral tannakian category. The quiver $\Q(k)$ is not exactly the one Nori used originally: the objects in Nori's original quiver are only those of the form $[X,Y,n, 0]$, and morphisms of type (c) are not considered. The linear hull of this subquiver is the category of \emph{effective motives}, and to compensate for morphisms of type (c) Nori inverts the motive $H^1(\IG_m)$ with respect to a monoidal structure, which is not easy to construct. We learnt the idea to build in Tate twists from the start from Ayoub \cite[p. 6]{Ayo14}. 
\end{par}
\end{para}

\vspace{4mm}
\begin{para}\label{Par:Sum101IsElementary}
Morphisms in Nori's quiver induce morphisms of motives. In particular, homotopy equivalences induce isomorphisms of motives. The morphism of pairs \eqref{Eqn:HReplacement} and the connecting morphism in the long exact sequence of the triple \eqref{Eqn:ConeSuspensionTriple} yield an isomorphism 
\begin{equation}\label{Eqn:DegreeShift}
H^n(X,Y)(i) \cong H^{n+1}(\Sigma(X,Y))(i)
\end{equation}
and morphisms of type (c) yield isomorphisms
\begin{equation}\label{Eqn:TwistShift}
H^n(X,Y)(i) \cong H^{n+1}((\IG_m,\{1\})\wedge (X,Y))(i+1).
\end{equation}
We can now fulfill our promise made in the introduction and explain why a sum of motives
$$\bigoplus_\alpha H^{n_\alpha}(X_\alpha,Y_\alpha)(i_\alpha)
$$
is isomorphic to a motive $H^{n_0}(X_0,Y_0)(i_0)$. Indeed, by using \eqref{Eqn:DegreeShift} and \eqref{Eqn:TwistShift} repeatedly, we may assume that in the sum all the degrees $n_\alpha$ and all the twists $i_\alpha$ are equal. Once we have arranged this, it suffices to take for $(X_0,Y_0)$ the wedge product of the $(X_\alpha,Y_\alpha)$. 
\end{para}

\vspace{4mm}
\begin{para}
Let $\Qaff(k) \subseteq \Q(k)$ be the full subquiver of $\Q(k)$ whose objects are those tuples $[X,Y,n,i]$ where $X$ is affine and $Y$ is non-empty. Let $\Q_h(k)$ be the quiver whose objects are the same as those of $\Qaff(k)$ but whose morphisms of type (a) are given by morphisms in the category $\bHo(k)$ instead. We can restrict the quiver representation $\rho\colon \Q(k) \to \bVec_\IQ$ to a quiver representation $\rho_\mathrm{aff}$ of $\Qaff(k)$, which then factorises over a representation $\rho_\mathrm{h}$ of $\Qh(k)$. The situation is summarised in the following commutative diagram of quiver morphisms and representations: 
\begin{equation}\label{Eqn:EndQEndQaffEndQhDiagram}
\begin{diagram}
\node{\Q(k)}\arrow{se,b}{\rho}\node{\Qaff(k)}\arrow{s,r}{\rho_\mathrm{aff}}\arrow{w,t}{\supseteq}\arrow{e,t}{\mathrm{can.}}\node{\Qh(k)}\arrow{sw,b}{\rho_\mathrm{h}}\\
\node[2]{\bVec_\IQ}
\end{diagram}
\end{equation}
\end{para}

\vspace{4mm}
\begin{prop}\label{Eqn:QuiverVariants1}
The morphisms of proalgebras 
$\End(\rho) \xrightarrow{\quad} \End(\rho_\mathrm{aff}) \xleftarrow{\quad} \End(\rho_\mathrm{h})$
obtained from the commutative diagram \eqref{Eqn:EndQEndQaffEndQhDiagram} are isomorphisms.
\end{prop}

\begin{proof}
\begin{par}
Let $e$ be an endomorphism of $\rho$. Concretely, $e$ is a collection of linear endomorphisms $e_q \colon \rho(q)\to \rho(q)$, one for every object $q = [X,Y,n,i]$ of $\Q(k)$, whhich are compatible in that, for every morphism $f\colon p\to q$ in $\Q(k)$, the diagram of vector spaces
\begin{equation}\label{Eqn:QuiverVariants1BasicSquare}
\begin{diagram}
\node{\rho(p)}\arrow{e,t}{\rho(f)}\arrow{s,l}{e_p}\node{\rho(q)}\arrow{s,r}{e_q}\\
\node{\rho(p)}\arrow{e,t}{\rho(f)}\node{\rho(q)}
\end{diagram}
\end{equation}
commutes. In order to show that the morphism 
\begin{equation}\label{Eqn:QuiverToAffine}
    \End(\rho) \xrightarrow{\quad} \End(\rho_{\mathrm{aff}})
\end{equation}
is injective, we may assume that $e_q=0$ for all $q=[X,Y,n,i]$ where $X$ is affine, and have to show that $e$ is zero. This is an immediate consequence of Jouanolou's trick: given an arbitrary object $[X,Y,n,i]$ of $\Q(k)$, there exists an affine variety $\widetilde X$ and a homotopy equivalence $f\colon  \widetilde X \to X$. Setting $\widetilde Y = f^{-1}(Y)$ and $\widetilde q = [\widetilde X, \widetilde Y,n,i]$, the morphism $f\colon \widetilde q \to q$ of type (a) induces an isomorphism of vector spaces $\rho(f)\colon \rho(\widetilde q) \to \rho(q)$. In case $\widetilde Y$ is empty, add to $\widetilde X$ and isolated point $\ast$ and set $\widetilde Y = \ast$. The diagram \eqref{Eqn:QuiverVariants1BasicSquare} for this particular morphism $f$, and hence $e_q$ is indeed zero. Let us for now refer to $\widetilde q \to q$ as an \emph{affine homotopy replacement}. In order to show that the morphism of proalgebras \eqref{Eqn:QuiverToAffine} is surjective as well, consider an element $e$ of $\End(\rho_{\mathrm{aff}})$. Choosing arbitrary affine homotopy replacements as before, we can extend $e$ to a well defined collection of endomorphisms $e_q$ for all objects $q$ of $\Q(k)$. We must check that the square \eqref{Eqn:QuiverVariants1BasicSquare} commutes, knowing that such squares commute for morphisms in $\Qaff(k)$. This can be checked by observing that for every morphism $f\colon p\to q$ in $\Q(k)$, we can choose affine homotopy replacements $\widetilde p \to p$ and $\widetilde q \to q$ such that there is a morphism $\widetilde f\colon  \widetilde p\to \widetilde q$ in $\Qaff(b)$ for which the diagram
$$\begin{diagram}
\node{\rho(\widetilde p)}\arrow{e,t}{\rho(\widetilde f)}\arrow{s,l}{\cong}\node{\rho(\widetilde q)}\arrow{s,r}{\cong}\\
\node{\rho(p)}\arrow{e,t}{\rho(f)}\node{\rho(q)}
\end{diagram}$$
commutes. Attaching this square to \eqref{Eqn:QuiverVariants1BasicSquare} yields a cube where all faces other than \eqref{Eqn:QuiverVariants1BasicSquare} commute, so  \eqref{Eqn:QuiverVariants1BasicSquare} commutes as well. Thus, the collection $(e_q)_{q\in \Q(k)}$ is an element of $\End(\rho)$ as we wanted to show. 
\end{par}

\begin{par}
The isomorphism $\End(\rho_\mathrm{aff})\cong  \End(\rho_\mathrm{h})$ is a formality: since the quivers $\Qaff(k)$ and $\Qh(k)$ have the same objects, we just need to check that a collection of endomorphisms $e_q$ commutes with morphisms in $\Qaff(k)$ if and only if it commutes with morphisms in $\Qh(k)$, but that is obvious.
\end{par}
\end{proof}

%%%%%%%%%%%%%%%%%%%%%%%%%%%%%%%%%%%%%%%%%%%%%%%%%%%%%%%%%%
\vspace{14mm}%%%%%%%%%%%%%%%%%%%%%%%%%%%%%%%%%%%%%%%%%%%%%
\section{The subquotient question for Nori's motives}%%%%%
%%%%%%%%%%%%%%%%%%%%%%%%%%%%%%%%%%%%%%%%%%%%%%%%%%%%%%%%%%

\begin{par}
In this section, we prove the main theorem from the introduction. The technical difficulty that remains to surmount is to show that each motive over $k$ is a module over the endomorphism ring of a finite subquiver of $\Qh(k)$ containing only morphisms of type (a). Proposition \ref{Pro:SingleObjectQuiverFactorisation} states that this is indeed possible; in fact, we can write $M$ as a module over the endomorphism ring of a subquiver of $\Qh(k)$ with only one object and possibly many endomorphisms, necessarily of type (a). We fix for the whole section a field $k\subseteq \IC$ and convene that all quivers are subquivers of $\Qh(k)$ equipped with the restriction of the standard representation $\rho = \rho_{\mathrm h}\colon \Qh(k) \to \bVec_\IQ$. 
\end{par}

\vspace{4mm}
\begin{defn}
Let $Q_0$ and $Q_1$ be finite subquivers of $\Qh(k)$. We say that $Q_0$ and $Q_1$ are \emph{equivalent} if there exists a finite subquiver $Q^+$ of $\Qh(k)$ containing $Q_0$ and $Q_1$ and an isomorphism of $\End(\rho|_{Q^+})$ algebras
$$\End(\rho|_{Q_0}) \xrightarrow{\:\:\cong\:\:} \End(\rho|_{Q_1}).$$
\end{defn}

\vspace{4mm}
\begin{prop}\label{Pro:SingleObjectQuiverFactorisation}
Every finite subquiver of $\Qh(k)$ is equivalent to a quiver containing only one object.
\end{prop}

\begin{proof}
This follows from Lemma \ref{Lem:EliminatingTypesbc} and Lemma \ref{Lem:ReductionToSingleObject}.
\end{proof}

\vspace{4mm}
\begin{para}\label{Par:SingleObjectQuiverFactorisation}
An alternative way of defining equivalence of quivers would be to say that $Q_0$ and $Q_1$ are equivalent if the kernels of the restriction maps
$$\End(\rho) \xrightarrow{\:\:\mathrm{res}\:\:} \End(\rho|_{Q_0})\qquad \End(\rho) \xrightarrow{\:\:\mathrm{res}\:\:} \End(\rho|_{Q_1})$$
are equal. Here is how we will use Proposition \ref{Pro:SingleObjectQuiverFactorisation}: given a motive $M$, that is, a finite-dimensional continuous $\End(\rho)$-module, we may by definition of continuity find some finite subquiver $Q$ of $\Qh(k)$ such that the structural map $\End(\rho) \to \End_\IQ(M)$ factors through the finite-dimensional algebra $E = \End(\rho|_Q)$. By Proposition \ref{Pro:SingleObjectQuiverFactorisation}, we may arrange $Q$ to contain only one object, say $q = [X,Y,n,i]$. As an $E$-module, $M$ is a quotient of $E^n$ for some $n\geq 0$, so in order to prove that $M$ is a quotient of a motive of the form $H^{n_0}(X_0,Y_0)(i_0)$ it suffices to show that $E$ is. As an $E$-module, $E$ is a submodule of $$\End(\rho(q)) \cong H_n(X,Y)(-i) \otimes H^n(X,Y)(i)$$
where in the tensor product $E$ acts via $e(u\otimes v) = u\otimes ev$. The commutator of an endomorphism $f$ of $q$ in $\End(\rho(q))$ is the kernel of the map $f_\ast \otimes\id - \id\otimes f^\ast$, and we will then use cogroup structures on suspensions to show that the intersection of finitely many such kernels has the desired form.
\end{para}

\vspace{4mm}
\begin{para}\label{Par:CloningProcess}
\begin{par}
Let $\rho|_Q: Q \to \bVec_{\IQ}$ be the standard representation on some subquiver $Q$ of $\Qh(k)$, and suppose that we dislike certain objects in $Q$, and want to replace them with more likeable objects, without changing the endomorphism algebra $\End(\rho|_Q)$. In other words, we wish to find 
a quiver $Q_0$ which is equivalent to $Q$ and contains only likeable objects. That may be possible in theory, as follows. Write $Q_\bad$ for the full subquiver of bad objects and $Q_\gud$ for the full subquiver of good objects of $Q$, and let us enlarge $Q$ to a quiver $Q^+$ in three steps. 
\end{par}
\begin{par}\underline{Step 1}: Start with setting $Q^+ = Q$. Then, find for each bad object $q$ of $Q_\mathrm{b}$ a finite, connected quiver $L(q)$ containing $q$ and also containing a non-empty connected subquiver $L_\gud(q)$ consisting of good objects, such that the diagram of vector spaces $\rho(L(q))$ is a commutative diagram of isomorphisms. For an object $q'$ in $L(q)$, denote by $\lambda(q')$ the isomorphism $\rho(q) \to \rho(q')$ appearing in $\rho(L(q))$. We add to $Q^+$ these quivers $L(q)$. We understand here that we have made sure that the only object common to $L(q)$ and $Q$ is $q$, and that for different objects $p$ and $q$ in $Q_\bad$, the quivers $L(p)$ and $L(q)$ are disjoint.
\end{par}
\begin{par}\underline{Step 2}:
Next, for every morphism $f\colon  p\to q$ in $Q_\bad$, find and add to $Q^+$ a morphism $f':p'\to q'$, where $p'$ and $q'$ are objects in $L_\gud(p)$ and $L_\gud(q)$, such that the diagram
\begin{equation}\label{Eqn:DgeCloneStep2}
\begin{diagram}
\node{\rho(p)}\arrow{s,l}{\lambda(p')}\arrow{e,t}{\rho(f)}\node{\rho(q)}\arrow{s,r}{\lambda(q')}\\
\node{\rho(p')}\arrow{e,t}{\rho(f')}\node{\rho(q')}
\end{diagram}    
\end{equation}
commutes.
\end{par}
\begin{par}\underline{Step 3}:
Finally, for every morphism $f\colon  p\to q$ or $f\colon q\to p$ between an object $q$ of $Q_\bad$ and an object $p$ of $Q_\gud$, find and add to $Q^+$ a morphism $f': p\to q'$ or $f':q'\to p$, where $q'$ is an object in $L_\gud(q)$, such that the corresponding diagram
\begin{equation}\label{Eqn:DgeCloneStep3}
\begin{diagram}
\node{\rho(p)}\arrow{se,b}{\rho(f')}\arrow{e,t}{\rho(f)}\node{\rho(q)}\arrow{s,r}{\lambda(q')}\node[2]{\rho(q)}\arrow{e,t}{\rho(f)}\arrow{s,l}{\lambda(q')}\node{\rho(p)}\\
\node[2]{\rho(q')}\node[2]{\rho(q')}\arrow{ne,b}{\rho(f')}
\end{diagram}
\end{equation}
commutes.
\end{par}
\begin{par} 
Denote now by $Q_\gud^+\subseteq Q^+$ the full subquiver of good objects, obtained from $Q^+$ by deleting all bad objects and all morphisms to and from bad objects. It is straightforward to check, as we will in \ref{Lem:DeletingClonesFromQuiverRep}, that the restriction morphisms
\begin{equation}\label{Eqn:CloneRestrictionIsomorphisms}
\End(\rho|_{Q^+_\gud}) \leftarrow \End(\rho|_{Q^+}) \to \End(\rho|_Q)    
\end{equation}
are isomorphisms of algebras. In particular, the quivers $Q^+_\gud$ and $Q$ are equivalent, and the quiver we were looking for at the beginning of this discussion is $Q_0=Q_\gud^+$.
\end{par}
\end{para}

\vspace{4mm}
\begin{lem}\label{Lem:DeletingClonesFromQuiverRep}
The quiver $Q_0$, as constructed in \ref{Par:CloningProcess}, is equivalent to $Q$.
\end{lem}

\begin{proof}
We check that the restriction morphisms \eqref{Eqn:CloneRestrictionIsomorphisms} are isomorphisms. Elements of $E^+ := \End(\rho|_{Q^+})$ are collections of linear endomorphisms $(e_q)_{q\in Q^+}$ indexed by objects of $Q^+$, with $e_q\in \End(\rho(q))$, satisfying
\begin{equation}\label{Eqn:CommutatorEquationForEndRhoQ}
e_q \circ \rho(f) = \rho(f)\circ e_p
\end{equation}
for each morphism $f\colon p\to q$ in $Q^+$. Elements of $E := \End(\rho|_Q)$ and $E_0 = \End(\rho|_{Q_0})$ are described similarly. In order to prove that the restriction map $E^+ \to E$ is injective, consider an element $e = (e_q)_{q\in Q^+}$ of $E^+$ such that $e_q=0$ for all $q\in Q$, fix an object $q' \in Q^+$, and let us show that $e_{q'}$ is zero. If $q'$ is not already an object of $Q$, then $q'$ is an object of $L(q)$ for some unique $q\in Q_\bad$. By definition, the diagram of vector spaces and linear maps
\begin{equation}\label{Eqn:DeletableIsomorphisms}
\begin{diagram}
\node{\rho(q)}\arrow{s,l}{e_q}\arrow[2]{e,tb}{\lambda(q')}{\cong}\node[2]{\rho(q')}\arrow{s,r}{e_{q'}}\\
\node{\rho(q)}\arrow[2]{e,tb}{\lambda(q')}{\cong}\node[2]{\rho(q')}
\end{diagram}
\end{equation}
commutes, and since $e_q=0$ we have indeed $e_{q'}=0$. This settles injectivity of the map $E^+\to E$, and injectivity of $E^+ \to E_0$ is shown similarly. In order to prove that the restriction map $E^+ \to E$ is also surjective, fix an element $(e_q)_{q\in Q}$ of $E$. We construct a tuple $(e_q)_{q\in Q^+}$ by considering as before for each $q'\in Q^+$ which is not already in $Q$ the unique map $\lambda(q')\colon \rho(q) \to \rho(q')$ and take for $e_{q'}$ the unique endomorphism of $\rho(q')$ for which the square \eqref{Eqn:DeletableIsomorphisms} commutes. It remains to pick a morphism $f$ in $Q^+$ and check that the relation \eqref{Eqn:CommutatorEquationForEndRhoQ} holds. If the target and the source of $f$ both belong to $Q$ then $f$ is a morphism in $Q$ and \eqref{Eqn:CommutatorEquationForEndRhoQ} holds by definition. If neither target nor source of $f$ belong to $Q$, then \eqref{Eqn:CommutatorEquationForEndRhoQ} holds because the square \eqref{Eqn:DgeCloneStep2} in Step 2 is supposed to commute, and if the target of $f$ but not the source belongs to $Q$, or the other way around, then the commutativity of \eqref{Eqn:DgeCloneStep3} implies \eqref{Eqn:CommutatorEquationForEndRhoQ}. Surjectivity of the restriction morphism $E^+ \to E_0$ is shown similarly.
\end{proof}

\vspace{4mm}
\begin{lem}\label{Lem:EliminatingTypesbc}
Let $Q$ be a finite subquiver of $\Qh(k)$. There exists a quiver $Q_0$ which is equivalent to $Q$ and such that there exist integers $n_0$ and $i_0$ such that all objects in $Q_0$ are of degree $n_0$ and twist $i_0$.
\end{lem}

\begin{proof}
\begin{par}
Recall that we refer to the integers $n$ and $i$ in an object $[X,Y,n,i]$ of $\Qh(k)$ as \emph{degree} and \emph{twist} respectively. Given integers $n$ and $i$ and a quiver $Q \subseteq \Qh(k)$, let us denote by $Q[n,i]$ the full subquiver of $Q$ consisting of objects with degree $n$ and twist $i$. Notice that $Q[n,i]$ only contains morphisms of type (a).
\end{par}

\begin{par}
Let $Q \subseteq \Qh(k)$ be a finite quiver containing objects with different twists and degrees. Following the process outlined in \ref{Par:CloningProcess}, we will show that there exists a finite quiver $Q_0$ which is equivalent to $Q$ and contains fewer different twists, and then continue inductively until there is only one twist left. We then proceed with a different construction, reducing the number of different degrees and not adding any new twists. This will eventually lead to a quiver which is equivalent to $Q$ and has only one twist and one degree. 
\end{par}

\begin{par}
Since $Q$ is finite, only finitely many of the quivers $Q[n,i]$ are non-empty. Choose $(n_0,i_0)$ large enough, such that whenever $Q[n,i]$ non-empty, then 
$(n_0,i_0) = (n+d+t,i+t)$ for non-negative integers $d$ and $t$. The quiver $Q$ can be drawn schematically as a finite diagram of the shape
$$
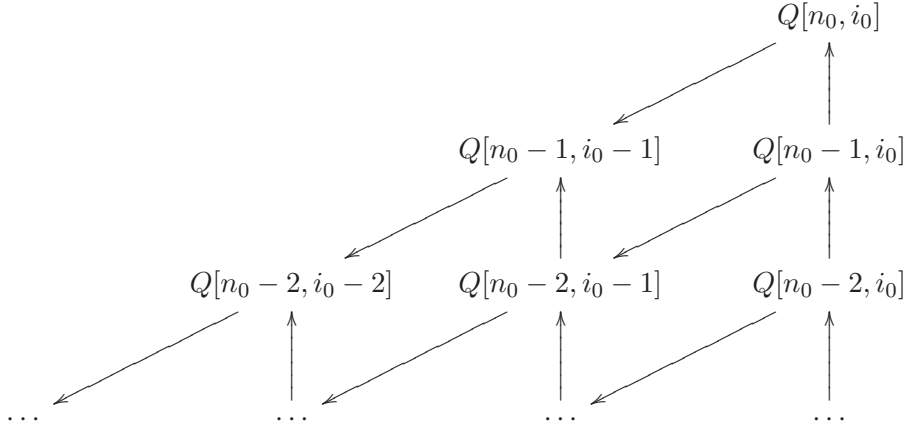
\begin{diagram}
\node[4]{Q[n_0,i_0]}\arrow{sw}\\
\node[3]{Q[n_0-1,i_0-1]}\arrow{sw}\node{Q[n_0-1,i_0]}\arrow{n}\arrow{sw}\\
\node[2]{Q[n_0-2,i_0-2]}\arrow{sw} \node{Q[n_0-2,i_0-1]} \arrow{sw} \arrow{n} \node{Q[n_0-2,i_0]}\arrow{sw} \arrow{n}\\
\node{\cdots}\node{\cdots}\arrow{n}\node{\cdots}\arrow{n}\node{\cdots}\arrow{n}
\end{diagram}$$
where the vertical arrows symbolise (many) morphisms of type (b) and the diagonal arrows symbolise morphisms of type (c), and where all nodes are finite quivers with internal morphisms only of type (a).
\end{par}

\vspace{2mm}
\begin{par}{\bf Claim: }
\emph{Let $i_1$ be the smallest integer such that $Q[n,i_1]$ is non-empty for some $n\leq n_0$, and suppose $i_1<i_0$. There exists a finite quiver $Q_0$ which is equivalent to $Q$ and such that if $Q_0[n,i]$ is non-empty, then $i_1<i\leq i_0$ and $n\leq n_0$.}
\end{par}

\begin{par}
To prove this claim, let us denote by $Q[i_1]$ the full subquiver of $Q$ of objects with twist $i_1$ and declare these to be the bad objects. We will construct $Q^+$ as outlined in \ref{Par:CloningProcess}. As for the first step, let $Q^+$ be the quiver, subject to further enlargement, obtained from $Q$ by adding for every $q = [X,Y,n,i_1]$ of $Q[i_1]$ the quiver $L(q)$ consisting of the two objects $q$ and
$$Tq= [X\times \IG_m, (Y\times\IG_m)\cup (X \times \{1\}),n+1,i_1+1]$$
and the morphism $\kappa_q\colon Tq \to q$ of type (c). The induced linear map $\rho(\kappa_q)\colon \rho(Tq)\to \rho(q)$ is an isomorphism. Objects in $Q^+$ have twists between $i_1$ and $i_0$, and degrees at most $n_0$. The construction of $Tq\to q$ is functorial in the evident way for morphisms $f\colon  p\to q$ in $Q[i_1]$ of type (a) and (b), so that the following diagram of vector spaces, corresponding to \eqref{Eqn:DgeCloneStep2} in the abstract setting, commutes.
$$\begin{diagram}
\setlength{\dgARROWLENGTH}{6mm}
\node{\rho(p)}\arrow{e,t}{\rho(f)} \arrow{s,l}{\cong}\node{\rho(q) }\arrow{s,r}{\cong}\\
\node{\rho(Tp)} \arrow{e,t}{\rho(Tf)}\node{\rho(Tq)}
\end{diagram}$$
As for the second step in the process, add the morphisms $Tf$ to $Q^+$. For the final step, whenever there is a morphism in $Q$ between an object $q$ of $Q[i_1]$ and an object $p$ not in $Q[i_1]$, this morphism must be a morphism $p\to q$ of type (c). Thus $p$ is a copy of $Tq$, and we add the canonical isomorphism $p=Tq$ to $Q^+$. Now we can define $Q_0 \subseteq Q^+$ to be the full subquiver obtained by deleting objects in $Q[i_1]$. As we have checked in Lemma \ref{Lem:DeletingClonesFromQuiverRep}, the quivers $Q$ and $Q_0$ are equivalent, and by construction $Q_0$ contains only objects with twist $i_1< i \leq i_0$ and degrees $n\leq n_0$. This proves the claim. Arguing by induction, we can continue the proof of the lemma under the assumption that $Q$ contains only objects with twist $i_0$.
\end{par}

\vspace{2mm}
\begin{par}{\bf Claim: }
\emph{Let $n_1$ be the smallest integer such that $Q[n_1,i_0]$ is non-empty, and suppose $n_1<n_0$. There exists a finite quiver $Q_0$ which is equivalent to $Q$ and such that, if $Q_0[n,i]$ is non-empty, then $i = i_0$ and $n_1<n\leq n_0$.}
\end{par}

\begin{par}
Let us denote by $Q[n_1] := Q[n_1,i_0]$ the full subquiver of $Q$ whose objects are those of degree $n_1$, and declare these to be the bad objects to be replaced. Given an object $q = [Y,Z,n_1,i_0]$ of $Q[n_1]$ let us denote by $Hq$ and $\Sigma q$ the objects
\begin{eqnarray*}
Hq & = & [(Y\times \{0,1\}) \cup (Z\times\IA^1), (Y \times\{0\}) \cup (Z\times\IA^1),n_1,i_0]\\
\Sigma q & = & [Y\times \IA^1, (Y \times\{0,1\}) \cup (Z\times\IA^1) ,n_1+1,i_0]
\end{eqnarray*}
and let us write $\iota_q\colon Hq \to q$ for the morphism of type (a), given by the inclusion of $Y = Y \times\{1\}$ into $Y\times \{0,1\} \cup (Z\times\IA^1)$ and $\delta_q\colon Hq \to \Sigma q$ for the unique morphism of type (b). The morphisms $\rho(\iota_q)\colon \rho(Hq) \to \rho(q)$ and $\rho(\delta_q)\colon \rho(Hq) \to \rho(\Sigma q)$ are isomorphisms, and their composite is the canonical isomorphism $H^{n_1}(Y,Z)(i_0) \cong H^{n_1+1}(\Sigma(Y,Z))(i_0)$ as we explained in \ref{Par:Puppe}. Let $Q^+$ be the quiver, subject to further enlargement, obtained by adding to $Q$  the objects and morphisms 
$$L(q) \quad=\quad\left[\Sigma q \xleftarrow{\:\:\delta_q\:\:} Hq\xrightarrow{\:\:\iota_q\:\:}  q\right]$$
for $q\in Q[n_1]$. The construction of the objects $Hq$ and $\Sigma q$ and morphisms $\delta_q$ and $\iota_q$ is in the obvious way functorial for morphisms $f\colon p\to q$ in $Q[n_1]$, which are all of type (a), and the following diagram of vector spaces corresponding to \eqref{Eqn:DgeCloneStep2} in the abstract setting commutes: 
$$\begin{diagram}
\setlength{\dgARROWLENGTH}{6mm}
\node{\rho(\Sigma p)}\arrow{s,l}{\rho(\Sigma f)}\node{\rho(Hp)}\arrow{w,t}{\cong} \arrow{e,t}{\cong}\arrow{s,l}{\rho(Hf)}\node{\rho(p)}\arrow{s,l}{\rho(f)}\\
\node{\rho(\Sigma q)}\node{\rho(Hq)}\arrow{w,t}{\cong}\arrow{e,t}{\cong}\node{\rho(q)}
\end{diagram}$$
As for the second step in \ref{Par:CloningProcess}, add for every morphism $f$ in $Q[n_1]$ the morphism $\Sigma f$ to $Q^+$. For the third and final step, whenever there is a morphism in $Q$ between an object of $Q[n_1]$ and an object not in $Q[n_1]$, this morphism must be a morphism $q\to p$ of type (b), say
$$d\colon [Y,Z,n_1,i_0] \to [X,Y,n_1+1,i_0]$$
of type (b). Add then to $Q^+$ the morphism $s^{-1}f\colon \Sigma q \to p$ of type (a) as given in \eqref{Eqn:PuppeConnectorRoof}. The commutative diagram \eqref{Eqn:PuppeConnectorInCohomology} cast in different notation is the following commutative triangle: 
$$\begin{diagram}
\node{\rho(q)}\arrow{s,l}{\partial}\arrow{e,t}{\rho(d)}\node{\rho(p)}\\
\node{\rho(\Sigma q)}\arrow{ne,b}{\rho(s^{-1}f)}
\end{diagram} \qquad \qquad \partial = \rho(\delta_q)\circ \rho(\iota_q)^{-1}$$
This completes step 3 in \ref{Par:CloningProcess}, and hence proves the claim. Arguing by induction on the number of different degrees in $Q$ finishes the proof of the lemma. 
\end{par}
\end{proof}

\vspace{4mm}
\begin{lem}\label{Lem:ReductionToSingleObject}
Let $Q$ be a finite subquiver of $\Qh(k)$ and suppose that there exist integers $n_0$ and $i_0$ such that all objects in $Q$ are of degree $n_0$ and twist $i_0$. There exists a quiver $Q_0$ which is equivalent to $Q$ and consists of only one object $q_0$ and endomorphisms.
\end{lem}

\begin{proof}
For notational convenience, we index objects of $Q$ by a finite set, $\mathrm{Obj}(Q) = (q_\alpha)_{\alpha \in A}$, and write $q_\alpha = [X_\alpha, Y_\alpha, n_0,i_0]$ for every $\alpha\in A$. Define $q_0=[X_0,Y_0,n_0,i_0]$ to be the object obtained from the pair of varieties
$$(X_0, Y_0) = \bigvee_{\alpha\in A} (X_\alpha, Y_\alpha)$$
and let us construct a quiver $Q^+$ by adding to $Q$ the object $q_0$ and the following morphisms:
\begin{enumerate}
    \item For each $\alpha\in A$, the morphism $q_0\to q_\alpha$ given by the inclusion $\iota_\alpha\colon X_\alpha \to X_0$.
    \item For each $\alpha\in A$, the morphism $q_\alpha \to q_0$ given by the morphism $\pi_\alpha\colon X \to X_\alpha$ which is the identity on $X_\alpha$ and the zero map on all other components. 
    \item For each morphism $h\colon q_\alpha \to q_\beta$ in $Q$ the endomorphism $q_0\to q_0$ given by the composite $\iota_\beta \circ h \circ \pi_\alpha$.
\end{enumerate}
The vector space $\rho(q_0)$ is the direct sum
$$\rho(q_0) = \bigoplus_{\alpha\in A} \rho(q_\alpha)\:,$$
the morphisms $\rho(\iota_\alpha)\colon \rho(q_0) \to \rho(q_\alpha)$ are projections and the morphisms \hbox{$\rho(\pi_\alpha)\colon \rho(q_\alpha) \to \rho(q_0)$} are the inclusions. The endomorphisms of the object $q_0$ induce, besides the identity, the linear endomorphisms
$$\rho(q_0) \xrightarrow{\:\:\mathrm{proj.}\:\:} \rho(q_\alpha) \xrightarrow{\:\:\rho(h)\:\:}\rho(q_\beta) \xrightarrow{\:\:\mathrm{incl.}\:\:} \rho(q_0)$$
for every morphism $h$ of $Q$, in particular projectors are obtained from identity morphisms $\id_{q_\alpha}$. It is clear that to give an endomorphism of $\rho(q_0)$ which commutes with all these linear endomorphisms is the same as to give an endomorphism of the representation \hbox{$\rho|_Q�\colon Q\to \bVec_\IQ$,} or more precisely, that the algebra morphisms
$$\End(\rho|_Q) \leftarrow \End(\rho|_{Q^+}) \to \End(\rho|_{Q_0})$$
are isomorphisms, which is what we wanted to show.
\end{proof}

\vspace{4mm}
\begin{par}
The conjunction of the statement of Lemma \ref{Lem:EliminatingTypesbc} and Lemma \ref{Lem:ReductionToSingleObject} yields Proposition \ref{Pro:SingleObjectQuiverFactorisation}, stating that every finite subquiver of $\Qh(k)$ is equivalent to a quiver containing only one object. We now follow the outline \ref{Par:SingleObjectQuiverFactorisation} towards the proof of our main theorem.
\end{par}

\vspace{4mm}
\begin{prop}\label{Pro:KernelIsElementaryQuotient}
Let $f_\alpha\colon (X_\alpha,Y_\alpha) \to (X,Y)$ be morphisms in $\bHo(k)$. There exists a morphism $f\colon (X,Y) \to (X_1,Y_1)$ in $\bHo(k)$ such that, for every integer $n\geq 0$, the sequence
\begin{equation}\label{Eqn:KernelIsElementaryQuotient}
H^n(X_1,Y_1)\xrightarrow{\:\: f^\ast\:\:} H^n(X,Y) \xrightarrow{\:\: (f^\ast_\alpha) \:\:}  \prod_\alpha H^n(X_\alpha, Y_\alpha)    
\end{equation}
is exact. In other words, the image of $f^\ast$ is the intersection of the kernels of the maps $f_\alpha^\ast$.
\end{prop}

\begin{proof}
Since $H^n(X,Y)$ is finite-dimensional, it suffices to consider the case where we are given finitely many morphisms $f_\alpha$. We set
$$(X_0, Y_0) = \bigvee_\alpha (X_\alpha, Y_\alpha)$$
and write $f_0\colon (X_0,Y_0)\to (X,Y)$ for the map given by $f_\alpha$ on the component $X_\alpha$. The linear map $f_0^\ast$ induced by $f_0$ in cohomology is identical to the morphism $(f_\alpha)$ in \eqref{Eqn:KernelIsElementaryQuotient}, hence instead of the finite family of morphisms $f_\alpha$ we may consider the single morphism $f_0$. By Lemma \ref{Lem:ImmersionHomotopyFactorisation}, we may assume that $f_0$ is given by a closed immersion, so we may as well pretend that the morphism $f_0\colon X_0\to X$ is the inclusion of a closed subvariety and think of $Y_0$ as a closed subvariety of $X$ as well. We retain the following diagram of inclusions: 
$$\begin{diagram}
\node{X_0}\arrow{e,tb}{\subseteq}{f_0}\node{X}\\
\node{Y_0}\arrow{e,t}{\subseteq}\arrow{n,l}{\subseteq}\node{Y}\arrow{n,l}{\subseteq}
\end{diagram}$$

\begin{par}
In the special case where $Y_0$ is the intersection of $X_0$ and $Y$, the proposition is immediate. Indeed, in that case there is a long exact sequence
$$\cdots\to H^n(X,X_0\cup Y) \to H^n(X, Y) \xrightarrow{f_0^\ast}  H^n(X_0,Y_0) \to H^{n+1}(X,X_0\cup Y)\to \cdots$$
where morphisms between cohomology groups of the same degree are induced by inclusions. The pairs $(X_1,Y_1) := (X,Y\cup X_0)$ and the morphism $f\colon (X,Y) \to (X_1, Y_1)$ given by the identity of $X$ is what we were searching for. 
\end{par}

\begin{par}
We reduce the general case to this special case. By choosing finitely many generators of the ideal of $\mathcal O_X(X)$ which defines the closed subvariety $Y_0\subseteq X,$ we obtain a function $\alpha \colon X \to \IA^N$ such that $\alpha^{-1}(0)=Y_0$. Write $\alpha_0$ for the restriction of $\alpha$ to $X_0$. We obtain the following diagram of closed immersions: 
$$\begin{diagram}
\node{\mathrm{Graph}(\alpha_0)}\arrow{e}\node{X \times \IA^N}\\
\node{Y_0\times\{0\}}\arrow{n}\arrow{e}\node{Y \times \{0\}}\arrow{n}
\end{diagram}$$
which is indeed a cartesian square: The intersection of the graph of $\alpha_0$ and $Y\times \{0\}$ is $Y_0\times\{0\}$. By the previously solved case, there exists a pair $(X_1,Y_1)$ and a morphism $h\colon X \times \IA^N \to X_1$ compatible with subvarieties, such that in the following diagram the upper row is exact: 
$$\begin{diagram}
\node{H^n(X_1,Y_1)}\arrow{se,b}{f^\ast} \arrow{e,t}{h^\ast} \node{\!H^n(X\!\!\times\!\!\IA^N,Y\!\!\times\!\!\{0\})}\arrow{e}\node{\!H^n(\mathrm{Graph}(\alpha_0),Y_0\!\!\times\!\!\{0\})}\\
\node[2]{H^n(X,Y)}\arrow{e,t}{f_0^\ast}\arrow{n,l}{\cong}\node{H^n(X_0,Y_0)}\arrow{n,l}{\cong}
\end{diagram}$$
The vertical morphisms are induced by the projections $X\times\IA^N \to X$ and $\mathrm{Graph}(\alpha_0)\to X_0$, so the square in the diagram commutes. These projections are homotopy equivalences, so the vertical maps in the diagram are indeed isomorphisms. A homotopy inverse to the projection $X_1\times\IA^N \to X_1$ is the inclusion $X_1 = X_1\times\{0\} \to X_1\times\IA^N$. Hence if we choose for $f$ the composite of $h$ with this inclusion, the whole diagram commutes, and the image of $f^\ast$ equals the kernel of $f_0^\ast$. 
\end{par}
\end{proof}

\vspace{4mm}
\begin{para}\label{Par:VectorSpaceAndMotiveProduct}
\begin{par}
Before we continue with the proof of the main theorem, let us clarify the meaning of tensor products between motives and vector spaces. First of all, given a motive $M$, that is, a vector space with a continuous action of the proalgebra $\End(\rho_{\mathrm h})$, and a vector space~$V$, we can let $E$ act on the second factor in $V \otimes M$. That means
$$e(v\otimes m)= v\otimes em$$
for $e\in E$, $v\in V,$ and $m\in M$. This way we can see $V \otimes M$ as a motive. The proalgebra $E = \End(\rho)$ has the additional structure of a Hopf algebra, hence in particular comes equipped with a comultiplication $E \to E \otimes E$ and a counit $E\to \IQ$, and we can use these structures to define an $E$-module structure on $V$ and on $V \otimes M$. These two $E$-module structures on $V \otimes M$ coincide, because the composition 
$$E \xrightarrow{\:\:\mathrm{comult.}\:\:} E \otimes E \xrightarrow{\:\:\mathrm{aug.}\otimes \id_E\:\:} \IQ \otimes E = E$$
is the identity morphism.
\end{par}

\begin{par}
Let $(X,Y)$ be a pair of algebraic varieties and fix an integer $n\geq 0$. We consider the vector spaces
$$M = H^n(X,Y)(i) \qquad V = \Hom_\IQ(M, \IQ)$$
so we can identify the vector space $V \otimes M$ with the vector space of $\IQ$-linear endomorphisms of $M= H^n(X,Y)(i)$. By its definition, the proalgebra $E$ acts continuously on $M$, this is indeed how we see $M$ as a motive in the first place, hence we obtain an $E$-module structure on $V\otimes M\cong \End_\IQ(M)$ by letting $E$ act on the second factor. On the other hand, the $E$-action on $M$ \emph{is} an algebra morphism $E \to \End_\IQ(M)$, and we can see $\End_\IQ(M)$ also this way as an $E$-module via left multiplication. Again, these two $E$-module structures on $V \otimes M$ coincide by basic algebra.
\end{par}
\begin{par}
Let $f$ be endomorphisms of $(X,Y)$ in $\bHo(k)$. The endomorphism $f$ induces a linear emdomorphism of $M$ which we denote by $f^\ast \colon M\to M$. The linear dual of $f^\ast$ is an endomorphism of $V$, which we denote by $f_\ast\colon V\to V$. Let  
\begin{equation}\label{Eqn:CommutatorAlgebraInclusion}
E(f) \subseteq \End_\IQ(M) = V\otimes M    
\end{equation}
be the commutator of $f^\ast$. As a subspace of $V\otimes M$, the commutator of $f^\ast$ is the subspace
$$E(f) = \Big\{\sum_iv_i \otimes m_i \:\Big|\:\: \sum_iv_i\otimes f^\ast(m_i) = \sum_if_\ast(v_i) \otimes m_i\Big\}$$
that is, the equaliser of the maps $\id_V\otimes f^\ast$ and $f_\ast\otimes \id_M$. By definition, the action of $E$ on $M$ factors over $E(f)$. 
\end{par}
\begin{par}
Finally, let us add to this discussion that we can use the antipode map $E\to E$ to define an $E$-module structure on $V$, and then the comultiplication to define an $E$-module structure on $V \otimes M$. This structure is very different from the one we introduced before, and we will not use it here.
\end{par}
\end{para}

\vspace{4mm}
\begin{prop}\label{Pro:CommutatorsAreElementary}
Let $(X,Y)$ be a pair of affine algebraic varieties, and let $n\geq 0$ be an integer. Denote by $M$ the motive $H^n(X,Y)$, and by $V$ the vector space 
$$V = H_n(X(\IC), Y(\IC); \IQ)$$
so $V$ is the linear dual of the vector space $H^n(X,Y)$. There exists an algebraic cogroup $(X_+, Y_+)$, and an isomorphism of motives
$$H^{n+1}(X_+,Y_+) \xrightarrow{\:\:\beta\:\:} V \otimes M$$
such that, for every endomorphism $f$ of $(X,Y)$ in $\bHo(k)$, there is an endomorphism $f_+$ of $(X_+, Y_+)$  in $\bHo(k)$ making the following diagram of motives commute: 
$$\begin{diagram}
\node{H^{n+1}(X_+,Y_+)}\arrow{s,l}{\beta}\arrow[2]{e,t}{f_+^\ast}\node[3]{H^{n+1}(X_+,Y_+)}\arrow{s,l}{\beta}\\
\node{V \otimes M}\arrow[3]{e,t}{f_\ast \otimes \id_M - \id_V\otimes f^\ast}\node[3]{V \otimes M}
\end{diagram}$$
\end{prop}

\begin{proof}
The homology with integral coefficients of the pair $(X(\IC), Y(\IC))$ modulo torsion is a lattice $\Lambda$ in $V$. Any morphism $f\colon (X,Y) \to (X,Y)$ induces a linear map $f_\ast \colon V\to V$ respecting this lattice. As in Example \ref{Exa:WedgeOfCirclesCogroup}, let $X_0$ be a disjoint union of points $x_0, \ldots, x_d$ and set $Y_0 = \{x_0\}$. There exists then an isomorphism of vector spaces $\alpha \colon H^1(\Sigma(X_0,Y_0)) \to V$ and a morphism $f_0\colon \Sigma(X_0,Y_0)\to \Sigma(X_0,Y_0)$ such that the diagram of vector spaces or motives
$$\begin{diagram}
\node{H^1(\Sigma(X_0,Y_0))}\arrow{s,l}{\alpha}\arrow{e,t}{f_0^\ast}\node{H^1(\Sigma(X_0,Y_0))}\arrow{s,l}{\alpha}\\
\node{V}\arrow{e,t}{f_\ast}\node{V}
\end{diagram}$$
commutes. Define
$$(X_+,Y_+) := \Sigma(X_0,Y_0) \wedge (X,Y)$$
and observe that since the cohomology of $\Sigma(X_0,Y_0)$ is concentrated in degree $1$, there is an isomorphism 
$$\beta: H^{n+1}(X_+,Y_+) \xrightarrow{\quad}H^1(\Sigma(X_0,Y_0)) \otimes H^n(X,Y) \xrightarrow{\:\:\alpha\otimes\id\:\:} V \otimes M$$
given by the K\"unneth formula. It is an isomorphism of motives as we have explained in~\ref{Par:VectorSpaceAndMotiveProduct}. Combining the constructions of $f_0$ and $\beta$ yields a commutative square of motives
\begin{equation}\label{Eqn:PlusConstructionDiagram1}
\begin{diagram}
\node{\hspace{-18mm}H^{n+1}((X_0,Y_0)\wedge (X,Y))}\arrow{s,l}{\beta} \arrow[3]{e,t}{(f_0\wedge \id)^\ast-(\id\wedge f)^\ast}\node[3]{H^{n+1}((X_0,Y_0)\wedge (X,Y))\hspace{-18mm}}\arrow{s,l}{\beta}\\
\node{V\otimes M}\arrow[3]{e,t}{f_\ast \otimes \id_M - \id_V\otimes f^\ast}\node[3]{V\otimes M}
\end{diagram}
\end{equation}
It remains to observe that $(X_+,Y_+)$ is the suspension of $(X_0,Y_0) \wedge(X,Y)$, hence already has the structure of a cogroup. Setting $f_+ = (f_0\wedge \id)  - (\id\wedge f)$ yields the sought commutative diagram in the statement of the proposition.
\end{proof}

\vspace{4mm}
\begin{thm}\label{Thm:SubquotientProblemMainQuotients}
Let $M$ be a motive over $k$. There exists a pairs of varieties $(X_1, Y_1)$ and integers $n_1$ and $i_1$, and a surjective morphism of motives
\begin{equation}\label{Eqn:SubquotientProblemMainQuotients}
H^{n_1}(X_1,Y_1)(i_1) \xrightarrow{\qquad} M.    
\end{equation}
\end{thm}

\begin{proof}
According to Proposition \ref{Pro:SingleObjectQuiverFactorisation}, there exists a subquiver $Q$ of $\Qh(k)$ containing only one object $q = [X,Y,n,i]$ and possibly many morphisms $f\colon q\to q$, all given by morphisms $(X,Y) \to (X,Y)$ in $\bHo(k)$, and a factorisation
$$\End(\rho_{\mathrm h}) \xrightarrow{\:\:e\mapsto e_q\:\:}\End(\rho_{\mathrm h}|_Q) \to \End_\IQ(M)$$
of the structural map $\End(\rho_\mathrm{h}) \to \End_\IQ(M)$. Set $E = \End(\rho_{\mathrm h}|_Q)$. Since $M$ is a finite-dimensional as a vector space, there is a surjective $E$-linear map
$$E^d \to M$$
for some $d\geq0$. This $E$-linear map is also a morphism of $\End(\rho_{\mathrm h})$-modules, that is, of motives. Therefore, it suffices to prove the theorem in the case $M=E$. Let us now change notations and as in the setup of Proposition \ref{Pro:CommutatorsAreElementary}, write $M$ for the motive $H^n(X,Y)$ and $V$ for the vector space $H_n(X,Y)$, so that there is a canonical isomorphism $\End_\IQ(M) = V\otimes M$ of motives. According to Proposition \ref{Pro:CommutatorsAreElementary}, there exists a pair $(X_+, Y_+)$ and an isomorphism of motives $H^{n+1}(X_+,Y_+) \xrightarrow{\:\:\beta\:\:} V \otimes M$ such that for every morphism $f\colon (X,Y) \to (X,Y)$ in $Q$, there is a morphism $f_+:(X_+, Y_+)$  in $\bHo(k)$ such that the diagram
$$\begin{diagram}
\node{H^{n+1}(X_+,Y_+)(i)}\arrow{s,l}{\beta}\arrow[2]{e,t}{f_+^\ast}\node[3]{H^{n+1}(X_+,Y_+)(i)}\arrow{s,l}{\beta}\\
\node{V \otimes M(i)}\arrow[3]{e,t}{f_\ast \otimes \id_{M(i)} - \id_V\otimes f^\ast}\node[3]{V \otimes M(i)}
\end{diagram}$$
of motives commutes. The kernel of $f_\ast \otimes \id_M - \id_V\otimes f^\ast$ is the subalgebra of $\End_\IQ(M)$ consisting of those endomorphisms of $M$ which commute with the endomorphism $f^\ast$. Hence, via the isomorphism $\beta$, the motive $E$ is isomorphic to the intersection
$$E = \bigcap_f \ker(f_+^\ast)$$
running over all morphisms $f:q\to q$ in $Q$. By Proposition \ref{Pro:KernelIsElementaryQuotient}, this intersection is the image of a morphism of motives $H^{n+2}(X_1,Y_1)(i)\to H^{n+2}(X_+,Y_+)(i)$, which we may even choose to be induced by some morphism $(X_+,Y_+) \to (X_1,Y_1)$, hence $E$ is isomorphic to a quotient of $H^{n+1}(X_1,Y_1)(i)$.
\end{proof}

\vspace{4mm}
\begin{para}
Theorem \ref{Thm:SubquotientProblemMainQuotients} is one of the two claims in the main theorem stated in the introduction. The other claim is similar, except that we demand an injection in the opposite direction instead of the surjection \eqref{Eqn:SubquotientProblemMainQuotients}. It is conceivable that one could rewrite the proof of Theorem \ref{Thm:SubquotientProblemMainQuotients} in a dual setup, where instead of cogroups and suspensions one works with groups and loop spaces. There is no need for that, because the category of motives is tannakian, hence in particular there is a good notion of duals, and the dual of a motive of the elementary form $M= H^n(X,Y)(i)$ is again of this form. 
\end{para}

\vspace{4mm}
\begin{thm}\label{Thm:SubquotientProblemMainSubs}
Let $M$ be a motive over $k$. There exists a pair of varieties $(X_0, Y_0)$ and integers $n_0$ and $i_0$, and an injective morphism of motives
\begin{equation}\label{Eqn:SubquotientProblemMainSubs}
M \xrightarrow{\quad}H^{n_0}(X_0,Y_0)(i_0).   
\end{equation}
\end{thm}

\begin{proof}
Let $M^\vee$ be the motive dual to $M$. By Theorem \ref{Thm:SubquotientProblemMainQuotients} there exists a surjective morphism of motives
\begin{equation}\label{Eqn:DualMotiveAsQuotient}
H^{n_1}(X_1,Y_1)(i_1) \to M^\vee    
\end{equation}
for some pair of varieties $(X_1,Y_1)$. Using resolution of singularities, we may suppose that $X_1$ is smooth with a smooth compactification $\overline X$ of dimension $d$, and a strict normal crossings divisor $\overline Y_0+\overline Y_1$ on $\overline X$ such that $X_1 = \overline X \setminus \overline Y_0$ and $Y_1 = \overline Y_1 \setminus (\overline Y_0\cap \overline Y_1)$ holds. Set $X_0 = \overline X \setminus \overline Y_1$ and $Y_0 = \overline Y_0 \setminus (\overline Y_0\cap \overline Y_1)$. The dual of the motive $H^{n_1}(X_1,Y_1)(i_1)$ is given by
$$ H^{n_1}(X_1,Y_1)(i_1)^\vee = H^{2d-n_1}(X_0,Y_0)(d-i_1)$$
hence the map dual to the surjection \eqref{Eqn:DualMotiveAsQuotient} is an injection as we sought it.
\end{proof}

\vspace{4mm}
\begin{para}
\begin{par}
We finish with a remark about \emph{effectivity}. Let $\Qeff(k) \subseteq \Q(k)$ be the full subquiver consisting of objects of the form $[X,Y,n,0]$, and write $\rho_\eff$ for the restriction of $\rho \colon \Q(k) \to \bVec_\IQ$ to this subquiver. The category of \emph{effective} motives $\bM_\eff(k)$ is the category of continuous $\End(\rho_\eff)$-modules. From the restriction morphism $\End(\rho) \to \End(\rho_\eff)$ we obtain a canonical functor
$$\bM_\eff(k) \to \bM(k)$$
which is exact and faithful. Nori originally defines $\bM(k)$ as the category obtained from $\bM_\eff(k)$ by $\otimes$-inverting the Lefschetz object $\IQ(-1) = H^1(\IG_m, \{1\})$, as explained in \cite[9.1.3]{HMS17}. There are several questions about this canonical functor, in particular one does not know at present whether it is full or not, nor does one know if its essential image is stable under extensions. The proof of Theorem \ref{Thm:SubquotientProblemMainQuotients} works verbatim for effective motives, with some shortcuts in the proof of Lemma \ref{Lem:EliminatingTypesbc}. We cannot deduce the statement of Theorem \ref{Thm:SubquotientProblemMainSubs} for effective motives in the same way, because $\bM_\eff(k)$ is not rigid - the dual of an effective motive is rarely effective. One might of course try to dualise the whole paper, that is, introduce algebraic $H$-groups and work with motives as comodules for some coalgebra. For this to succeed, one first needs to come up with a reasonable notion of loop spaces modeled by algebraic varieties.
\end{par}

\begin{par}
One can of course wonder whether for an arbitrary quiver representation $\rho\colon Q\to \bVec$ every object of the category of finite dimensional $\End(\rho)$-modules is isomorphic to a subobject or to a quotient of a finite sum of modules of the form $\rho(q)$ for $q\in Q$. In general, this is not the case, even in the particular situation where the category of $\End(\rho)$-modules acquires from somewhere a tannakian structure and a notion of weights. Examples can be fabricated as follows: let $V$ be a faithful representation of a linear algebraic group $G$, and let $Q$ be the quiver whose objects are finite sums of representations of the form $V^{\otimes a} \otimes (V^\vee)^{\otimes b}$ and whose morphisms are $G$-equivariant linear maps. We regard the forgetful functor $\rho\colon Q \to \bVec$ as a quiver representation. One can show that the category of $\End(\rho)$-modules is equivalent to the category of $G$-representations, but for a general (non-reductive) group $G$ not every representation can be injected into a sum of representations of the form $V^{\otimes a} \otimes (V^\vee)^{\otimes b}$. 
\end{par}
\end{para}

\bibliographystyle{amsalpha}

\providecommand{\bysame}{\leavevmode\hbox to3em{\hrulefill}\thinspace}
\providecommand{\MR}{\relax\ifhmode\unskip\space\fi MR }
% \MRhref is called by the amsart/book/proc definition of \MR.
\providecommand{\MRhref}[2]{%
  \href{http://www.ams.org/mathscinet-getitem?mr=#1}{#2}
}
\providecommand{\href}[2]{#2}

\end{document}